\documentclass[a4paper,reqno, 11pt]{amsart}
\usepackage{amssymb,amsmath}
\usepackage{amsopn}
\usepackage{mathrsfs}
\usepackage{color}
\usepackage{mathtools}
\usepackage{wasysym}
\usepackage{vmargin}
\usepackage{pdfsync}
\usepackage{vmargin}
\usepackage{color}
\usepackage{amscd}
\usepackage{verbatim}
\usepackage{bbm}
\usepackage[capitalise]{cleveref}

\theoremstyle{plain} 
\newtheorem{prop}{Proposition}[section] 
\newtheorem{thm}[prop]{Theorem}

\newtheorem{lem}[prop]{Lemma}
\newtheorem{cor}[prop]{Corollary}
\theoremstyle{definition}
\newtheorem{defi}[prop]{Definition}
\newtheorem{rem}[prop]{Remark}

\newtheorem{ex}[prop]{Example}

\newcommand{\megane}[1]{{}}

\newcommand{\ioe}{\leq}
\newcommand{\soe}{\geq}
\newcommand{\wt}{\widetilde}

\newcommand{\1}{\mathrm{1~\hspace{-1.4ex}l}}

\DeclareMathOperator{\ad}{ad}

\DeclareMathOperator{\Id}{Id}
\DeclareMathOperator{\sign}{sign}
\DeclareMathOperator{\dom}{Dom}

\def\R{\mathbb{R}}
\def\C{\mathbb{C}}
\def\N{\mathbb{N}}

\def \O{\mathcal{O}}
\def \tild{\widetilde}

\title[Quadratic behaviors of the Schrödinger equation]{Quadratic behaviors of the 1D linear Schrödinger equation with bilinear control}
\date{}
\author[Mégane Bournissou]{Mégane Bournissou$^*$}
\thanks{$^*$Univ Rennes, CNRS, IRMAR - UMR 6625, F-35000 Rennes, France.}
\thanks{This work benefits from the support of ANR project TRECOS,
grant ANR-20-CE40-0009.}
\email{megane.bournissou@ens-rennes.fr}

\begin{document}

\begin{abstract}
We consider a 1D linear Schrödinger equation, on a bounded interval, with Dirichlet boundary conditions and bilinear control. We study its controllability around the ground state when the linearized system is not controllable and wonder whether the quadratic term can help to recover the directions lost at the first order. More precisely, in this paper, we formulate assumptions under which the quadratic term induces a drift which prevents the small-time local controllability (STLC) of the system in appropriate spaces. 

For finite-dimensional systems, quadratic terms induce coercive drifts in the dynamic, quantified by integer negative Sobolev norms, along explicit Lie brackets which prevent STLC.

In the context of the bilinear Schrödinger equation, the first drift, quantified by the $H^{-1}$-norm of the control, was already observed in \cite{BM14} and used to deny STLC with controls small in $L^\infty$. In this article, we improve this result by denying STLC with controls small in $W^{-1, \infty}$. 

Furthermore, for any positive integer $n$, we formulate assumptions under which one may observe a quadratic drift quantified by the $H^{-n}$-norm of the control and we use it to deny STLC in suitable spaces. 

\end{abstract}

\maketitle

%\tableofcontents

\section{Introduction}
\subsection{Description of the control system}
Let $T>0$. In this paper, we consider the 1D linear Schrödinger equation given by, 
\begin{equation}  \label{Schrodinger} \left\{
    \begin{array}{ll}
        i \partial_t \psi(t,x) = - \partial^2_x \psi(t,x) -u(t)\mu(x)\psi(t,x),  \quad &(t,x) \in (0,T) \times (0,1),\\
        \psi(t,0) = \psi(t,1)=0, \quad &t \in (0,T).
    \end{array}
\right.  \end{equation}
Such an equation arises in quantum physics to describe a quantum particle stuck in an infinite potential well and subjected to a uniform electric field whose amplitude is given by $u(t)$. The function $\mu : (0,1) \rightarrow \mathbb{R}$ depicts the dipolar moment of the particle. This equation is a bilinear control system where the state is the wave function $\psi$ such that for all time $\| \psi(t) \|_{L^2(0,1)} = 1$ and $u : (0,T) \rightarrow \mathbb{R}$ denotes a scalar control. 

\subsection{Functional setting}
Unless otherwise specified, in space, we will work with complex valued functions. The Lebesgue space $L^2(0,1)$ is equipped with the hermitian scalar product given by
$$\langle f,g\rangle := \int_0^1 f(x) \overline{g(x)}dx, \quad f, g \in L^2(0,1).$$ Let $\mathcal{S}$ be the unit-sphere of $L^2(0,1)$.
%%%%
The operator $A$ is defined by
\begin{equation*}
\dom(A):=H^2(0,1) \cap H^1_0(0,1), \quad A\varphi:=-\frac{d^2 \varphi}{dx^2}.
\end{equation*}
Its eigenvalues and eigenvectors are respectively given by
\begin{equation*}
\forall j \in \N^*, \quad \lambda_j:= (j \pi)^2 \quad \text{ and } \quad \varphi_j:=\sqrt{2} \sin(j \pi \cdot).
\end{equation*}
The family of eigenvectors $(\varphi_j)_{j \in \N^*}$ is an orthonormal basis of $L^2(0,1)$. We denote by,  
$$\forall j \in \N^*, \quad \psi_j(t,x):= \varphi_j(x) e^{-i \lambda_j t}, \quad \forall (t,x) \in \R \times [0,1],$$ 
the solutions of the Schrödinger equation \eqref{Schrodinger} with $u \equiv 0$ and initial data $\varphi_j$ at time $t=0$. When $k=1$, $\psi_1$ is called the ground state. We also introduce the normed spaces linked to the operator $A$, given by, for all $s \soe 0$, 
\begin{equation*}
H^s_{(0)}(0,1):=\dom(A^{\frac{s}{2}}) 
\quad 
\text{ endowed with the norm }
\quad 
\left\|
\varphi
\right\|_{H^s_{(0)}(0,1)} 
:= 
\left( 
\sum \limits_{j=1}^{+\infty} 
\left| 
j^s \langle  \varphi, \varphi_j\rangle  
\right|^2 
\right)^{\frac{1}{2}}.
\end{equation*}
%%%
Finally, for $T>0$ and $u \in L^1(0,T)$, the family $(u_n)_{n \in \N}$ of the iterated primitives of $u$ is defined by induction as, 
\begin{equation*}
u_0:=u \quad \text{ and } \quad \forall n \in \mathbb{N}, \ u_{n+1}(t) := \int_0^t u_n(\tau) d\tau, \quad t \in [0,T].
\end{equation*}

\subsection{Hypotheses on the dipolar moment}
Let us make precise the assumptions on the dipolar moment $\mu$ we shall consider in the following. In all this paper, we will consider a dipolar moment $\mu$ at least in $H^3( (0,1), \R)$. Let $n \in \N^*$. The functional setting under which the $n$-th quadratic obstruction occurs in the vicinity of the ground state for the Schrödinger equation \eqref{Schrodinger} is the following. 
%
%\smallskip \noindent
%\textbf{(H1)}$_{n}$ The function $\mu$ is in $H^{2n+1}((0,1), \mathbb{R})$ and its first $n-1$ odd derivatives vanish at $x=0$ and $x=1$ (no boundary conditions are asked when $n=1$).

\smallskip \noindent
\textbf{(H1)$_K$} There exists $K \in \mathbb{N}^*$ such that $\langle \mu \varphi_1, \varphi_K\rangle =0$.

\smallskip \noindent 
\textbf{(H2)}$_{K, n}$ The sequence 
$
\left(
c_j:=
\langle 
\mu 
\varphi_1, 
\varphi_j
\rangle  
\langle 
\mu 
\varphi_K, 
\varphi_j
\rangle
\right)_{j \in \N^*}
$
satisfies
\begin{align}
\label{conv_sum}
&\sum 
\limits_{j=1}^{+\infty}
j^{4n} | c_j | 
<
+
\infty, 
\\
\label{coeff_quad_nul}
%A^p_K:=
%(-1)^{p-1}
A^p_K
:=
(-1)^{p-1}
&\sum 
\limits_{j=1}^{+\infty} 
\left(\lambda_j -  \frac{\lambda_1+\lambda_K}{2} \right) 
( \lambda_K-\lambda_j)^{p-1} 
(\lambda_j -\lambda_1)^{p-1}
c_j 
=
0,
\
\forall p\in \{1, \ldots, n-1\},
\\
A^n_K:=
(-1)^{n-1}
&\sum 
\limits_{j=1}^{+\infty} 
\left(\lambda_j -  \frac{\lambda_1+\lambda_K}{2} \right) 
( \lambda_K-\lambda_j)^{n-1} 
(\lambda_j -\lambda_1)^{n-1}
c_j 
\neq 0
.
\label{coeff_quad_non_nul}
\end{align}
%For all $p=1, \ldots, n-1$, $A^p_K =0$ and $A^n_K \neq 0$ where for all  $p \in \mathbb{N}^*$, \begin{equation}
%\label{def_Apk}
%A^p_K := (-1)^{p-1} \sum \limits_{j=1}^{+\infty} \left(\lambda_j -  \frac{\lambda_1+\lambda_K}{2} \right) ( \lambda_K-\lambda_j)^{p-1} (\lambda_j -\lambda_1)^{p-1} \langle \mu \varphi_1, \varphi_j\rangle  \langle \mu \varphi_K, \varphi_j\rangle.
%\end{equation}
%
Notice that by \eqref{conv_sum}, all the series considered in \eqref{coeff_quad_nul} and \eqref{coeff_quad_non_nul} converge absolutely. Moreover, the existence of a function $\mu$ satisfying (H1)$_K$-(H2)$_{K, n}$ is proved in Appendix \ref{existence_mu}. 
%\begin{rem}
%\label{decay_rate_coeff}
%Under the smoothness assumption on $\mu$ (H1)$_n$, the Fourier coefficients of $\mu \varphi_1$ and $\mu \varphi_K$ %have the following decay rate
%\begin{equation*}
%\label{coeff}
%|\langle \mu \varphi_K, \varphi_j\rangle |
%\quad
%\text{ and }
%\quad
%|\langle \mu \varphi_1, \varphi_j\rangle | = \O \left( \frac{1}{j^{2n+1}} \right).
%\end{equation*}
%Thus, all the series considered in (H3)$_n$ are convergent. 
%\end{rem}

\begin{rem}
For smooth vector fields $X, Y$ in $C^{\infty}(\R^p, \R^p)$, the Lie bracket $[X,Y]$ is defined as the following smooth vector field: 
$
[
X,
Y]
(x)
:=
X'(x) Y(x)
- Y'(x) X(x).
$
We also define by induction
\begin{equation*}
\ad_{X}^0(Y) := Y 
\quad
\text{ and }
\quad 
\forall k \in \N, 
\quad
\ad_{X}^{k+1}(Y) := [X, \ad_{X}^{k}(Y)].  
\end{equation*}
With $f_0:= -i A$ and $f_1:= i \mu$, \eqref{coeff_quad_nul} and \eqref{coeff_quad_non_nul} can be written formally in terms of Lie brackets as 
\begin{equation*}
\forall p=1, \ldots, n-1, 
\
\langle  [\ad_{f_0}^{p-1}(f_1), \ad_{f_0}^{p}(f_1) ] \varphi_1,  \varphi_K\rangle  =0
\ 
\text{ and }
\ 
\langle  [\ad_{f_0}^{n-1}(f_1), \ad_{f_0}^{n}(f_1) ] \varphi_1,  \varphi_K\rangle  \neq 0.
\end{equation*}
These Lie brackets are exactly those along which the quadratic order adds a drift, denying $W^{2n-3, \infty}$-STLC (see \cref{def_STLC}) for finite-dimensional systems $\dot{x}=f_0(x)+uf_1(x)$ in \cite[Theorem 3]{BM18}.
The previous formal computations can become rigorous for instance when, for some $p \in \mathbb{N}^*$, $\mu$ is in $H^{2p}(0,1)$ with its $p-1$ first odd derivatives vanishing at $x=0$ and $1$. Indeed, in this case, we get 
$$A^p_K
= 
\frac{(-1)^{p-1}}{2} 
\langle  
[\ad_A^{p-1}(\mu), \ad_A^{p}(\mu)] \varphi_1,  
\varphi_K\rangle,$$
where the iterated Lie brackets are well-defined (for all $k \in \N^*$, to compute $\ad^k_A(\mu)$, one needs to check that $\ad^{k-1}_A(\mu)$ is in $\dom A$) and denote commutators of operators (see \cref{expr_via_LB} for more details). 
\end{rem}

At a heuristic level, assumptions (H1)$_K$-(H2)$_{K, n}$  correspond to the fact that, in the asymptotic of small controls, the solution $\psi$ of the Schrödinger equation \eqref{Schrodinger} satisfies
\begin{equation*}
\Im \langle \psi(T), \varphi_K e^{-i \lambda_1 T} \rangle \approx -A_K^n \| u_n \|^2_{L^2(0,T)}.
\end{equation*} 

\subsection{Main result}
First, we define the notion of small-time local controllability (STLC) used in this paper, stressing the regularity imposed on the control as it plays a crucial role in the validity of controllability results.  

\begin{defi}
\label{def_STLC}
Let $(E_T, \|\cdot\|_{E_T})$ be a family of normed vector spaces of real functions defined on $[0,T]$ for $T> 0$.
The system \eqref{Schrodinger} is said to be E-STLC around the ground state if there exists $s \in \mathbb{N}$ such that for every $T>0$, for every $\varepsilon>0$, there exists $\delta> 0$ such that for every $\psi_f \in \mathcal{S}$ with $\|\psi_f - \psi_1(T)\|_{H^s_{(0)}(0,1)} < \delta$,  there exists $u \in L^2((0,T),\R) \cap E_T$ with $\|u\|_{E_T} < \varepsilon$ such that the solution $\psi$ of \eqref{Schrodinger} associated to the initial condition $\varphi_1$ satisfies $\psi(T)=\psi_f$.  
\end{defi}
Often, to prove STLC, one can use the linear test: if the linearized system is controllable, one can hope to prove the STLC of the nonlinear system through a fixed-point theorem. When the linearized system misses some directions, one can wonder whether the following term in the expansion of the solution can help to recover controllability along the lost direction.

For the Schrödinger equation \eqref{Schrodinger}, since \cite{BM14}, it is known that when one of the coefficients $\langle \mu \varphi_1, \varphi_j \rangle$ vanishes, one can give explicit impossible motions in small time, due to a drift quantified by the $H^{-1}$-norm of the control which denies $L^{\infty}$-STLC. However, Beauchard and Marbach highlighted in \cite{BM18} that for control-affine systems in finite dimension, to observe the first drift, the optimal norm for the smallness assumption on the control is the $W^{-1, \infty}$ one. Hence, the goal was to improve the result of \cite{BM14} and show that the first drift can be used to deny $W^{-1, \infty}$-STLC instead. Moreover, we give assumptions on $\mu$ under which one may observe a quadratic drift quantified by any $H^{-n}$-norm of the control, allowing to deny $H^{2n-3}$-STLC. 

\begin{thm}
\label{obstruction_n} 
Let $(K, n) \in {\mathbb{N}^*}^2$. Assume that $\mu$ satisfies (H1)$_K$-(H2)$_{K, n}$ . 
%
%\begin{itemize}
%\item its first $n-1$ odd derivatives are zero at $x=0$ and $x=1$, 
%\item there exists $K \in \mathbb{N}^*$ such that $\langle \mu \varphi_1, \varphi_K\rangle =0$, for $p=1, \ldots, n-1$, $A^p_K =0$ and $A^n_K \neq 0$, where $A^p_K$ is defined as 
%\begin{equation}
%\label{def_Apk}
%A^p_K := (-1)^{p-1} \sum \limits_{j=1}^{+\infty} \left(\lambda_j -  \frac{\lambda_1+\lambda_K}{2} \right) ( \lambda_K-\lambda_j)^{p-1} (\lambda_j -\lambda_1)^{p-1} \langle \mu \varphi_1, \varphi_j\rangle  \langle \mu \varphi_K, \varphi_j\rangle , \quad p \in \mathbb{N}^*,
%\end{equation}
%
%\item \blue{there exists $J$ a finite subset of $\mathbb{N}^* - \{1 \}$ of cardinal $n$ such that for all $j \in J$, we have $\langle \mu \varphi_1, \varphi_j\rangle  \neq 0$.  }
%\end{itemize} 
If $n \soe 2$ (resp. $n=1$), then the Schrödinger equation \eqref{Schrodinger} is not $H^{2n-3}$-STLC (resp. $W^{-1, \infty}$-STLC). 

\smallskip \noindent 
More precisely, there exists $C>0$ such that for every 
%$A \in (0, | A^n_K)$
$A \in (0, \frac{|A^n_K |}{4})$, 
there exists $T^*> 0$ such that for every $T \in (0,T^*)$, there exists $\eta > 0$ such that for every $u \in H^{2n-3}(0,T)$ (resp. $u \in L^2(0,T)$) with $\|u\|_{H^{2n-3}(0,T)} \ioe \eta$ (resp. $\|u_1\|_{L^{\infty}(0,T)} \ioe \eta$), the solution $\psi$ of \eqref{Schrodinger} with initial data $\varphi_1$ satisfies  
\begin{equation}
\label{eq:coercivity_n}
- 
\sign(A^n_K)
\Im \langle \psi(T), \varphi_K e^{-i \lambda_1 T}\rangle 
\soe 
%\red{\frac{(-A+ \alpha_K^n A^n_K)}{4}}
A 
\|u_n\|^2_{L^2(0,T)} 
-C
\| ( \psi - \psi_1 )(T) \|^2_{L^2(0,1)}.
\end{equation}
\end{thm}

\begin{rem}
\cref{obstruction_n} entails that, under hypotheses (H1)$_K$-(H2)$_{K, n}$ , there exists $R>0$ such that the following targets cannot be reached by the solution of  \eqref{Schrodinger},
\begin{equation*}
\forall \delta \in (0,R), 
\quad 
\psi_f=\left( \sqrt{1-\delta^2}\varphi_1+i \sign(A^n_K) \delta \varphi_K \right) e^{-i \lambda_1 T}.
\end{equation*}
Indeed, if there exists a control $u$ such that $\psi(T)=\psi_f$, then \eqref{eq:coercivity_n} leads to 
\begin{equation*}
%-\delta \soe \frac{\sign(A^n_K) A^n_K}{8} \| u_n\|^2_{L^2} - 2C\delta^2
%\quad
%\text{ and thus }
%\quad 
- 
\delta 
\soe 
A 
\| u_n\|^2_{L^2} 
-
2
C
\delta^2, 
\end{equation*}
which doesn't hold for small $\delta$. 
\end{rem}

\begin{ex}
For $\mu(x)=x-\frac{1}{2}$, one can compute that $\langle \mu \varphi_1, \varphi_1 \rangle =0 $ and $A^1_1 =1$. 
%\begin{equation*}
%\langle \mu \varphi_1, \varphi_1 \rangle =0 
%\quad 
%\text{ and }
%\quad 
%A^1_1 =1,
%\end{equation*}
%Hence, $\mu$ satisfies (H1)$_1$-(H2)$_{1, 1}$  and \cref{obstruction_n} applies.  Thus, the negative STLC result, presented by Coron in \cite{C06}, for a quantum particle in a moving one-dimensional infinite square potential well, holds without taking the speed and the displacement of the well as states. Moreover, to deny STLC, the control does not need to be small in $L^{\infty}$ as in \cite{C06} but rather small in $W^{-1, \infty}$.
Hence, $\mu$ satisfies (H1)$_1$-(H2)$_{1, 1}$ and \cref{obstruction_n} applies.  Thus, for the negative STLC result presented by Coron in \cite{C06}, for a quantum particle in a moving one-dimensional infinite square potential well, the control does not need to be small in $L^{\infty}$ but rather small in $W^{-1, \infty}$.
\end{ex}

The main ideas of this paper are presented on a toy-model in finite dimension in \cref{sec:dim_finie}. Then, we recall the well-posedness of the Schrödinger equation in \cref{wellposedness}. In \cref{expansion_solution}, error estimates on the expansion of the solution are given. The coercivity of the second-order term of the expansion is studied in \cref{section:coercivity}. Finally, \cref{obstruction_n} is proved in \cref{proof_obstructions}.

\subsection{State of the art}

\paragraph{\textit{Local exact controllability results.}}
First,  Turinici in \cite{T00} deduced a negative control result for the Schrödinger equation \eqref{Schrodinger} from the work \cite{BMS82} of Ball, Marsden and Slemrod on the controllability of bilinear control systems. This result was later completed by Boussaid, Caponigro and Chambrion in \cite{BCC20}.

However, these negative results are due to an `unfortunate' choice of functional setting in which controllability doesn't hold. It may not be due to a deep non-controllability. Thus, exact local controllability results for 1D models have been proven under a more appropriate functional setting by Beauchard in \cite{B05, B08}, later improved by Beauchard and Laurent in \cite{BL10}. 

Morancey and Nersesyan also proved the controllability of one Schrödinger equation with a polarizability term \cite{MN14} and of a finite number of equations with one control \cite{M14, MN15}.  Puel \cite{P16} also proved the local exact controllability for a Schrödinger equation, in a bounded regular domain of dimension $N \leq 3$, in a neighborhood of an eigenfunction corresponding to a simple eigenvalue, for controls $u=u(t,x)$.  

\ \paragraph{\textit{Global approximate results.}} Chambrion, Mason, Sigalotti and Boscain proved in \cite{CMSB09} the approximate controllability of Schrödinger in $L^2$, using Galerkin approximations, under hypotheses later refined by Boscain, Caponigro, Chambrion and Sigalotti in \cite{BCCS12}. Similar results were proved for one \cite{BCC20} or a finite number of equations \cite{BCS14} in more regular Sobolev spaces. Such results can also stem from exact controllability results in infinite time \cite{NN12bis} or from variational arguments \cite{N09}.

\ \paragraph{\textit{About quadratic obstructions.}} When the linearized system is not controllable, the strategy of performing a power series expansion of the solution, presented in \cite[Chap. 8]{C07} for finite-dimensional control system, can be used to prove positive controllability results as in \cite{B08, BM14, Cer07, CC09, CC04} or stabilization results as in \cite{CE19, CR17, CRX17}.

It can also be used to deny STLC. In \cite{BM18}, the authors proved that, in finite dimension, for scalar-input differential systems, when the linear test fails, STLC cannot be recovered from the quadratic term. Indeed, the second-order term adds a drift quantified by the $H^{-n}$-norm of the control, along an explicit Lie bracket, denying $W^{2n-3, \infty}$-STLC for the nonlinear system. This phenomenon was already observed in infinite dimension, for a Schrödinger equation, by Coron in \cite{C06} and by Beauchard and Morancey in \cite{BM14}. Indeed, the authors proved a coercivity type inequality, similar to \eqref{eq:coercivity_n}, to explicit impossible motions in small time. In that case, those motions were given, at least formally, by the same Lie brackets as in \cite{BM18}. In \cite{M18}, for a Burgers equation, those Lie brackets vanish and don't lead to obstruction. However, STLC is still denied proving an inequality analogous to \eqref{eq:coercivity_n} but involving a fractional Sobolev norm of the control instead. In \cite{BM20}, obstructions caused by both quadratic integer and fractional drifts are proven on a scalar-input  parabolic equation. Finally, a similar result has also been proved on a KdV system, for boundary controls in \cite{CKN20} by Coron, Koenig and Nguyen.

\section{A finite-dimensional counterpart}
\label{sec:dim_finie}
Let $p \in \N^*$ and $H_0, H_1 \in \mathcal{M}_p(\mathbb{R})$ symmetric matrices. Consider Schrödinger control systems of the form
\begin{equation}
\label{ODE}
i X'(t) = H_0 X(t) -u(t) H_1X(t),
\end{equation}
where the state is $X(t) \in \mathbb{C}^p$ and the control is $u(t) \in \R$. We write $(\varphi_1, \ldots, \varphi_p)$ for an orthonormal basis of eigenvectors of $H_0$, $(\lambda_1, \ldots, \lambda_p)$ for its eigenvalues and denote by $X_{1}(t):= \varphi_1 e^{-i \lambda_1 t}$. We take $X(0)=\varphi_1$ as initial data for \eqref{ODE}. In this section, the commutator of  $H_0$ and $H_1$ is denoted by $[H_0, H_1]:=H_0 H_1 - H_1 H_0$ and $\langle \cdot, \cdot \rangle$ denotes the classical hermitian scalar product on $\C^p$. 

\begin{thm}
\label{obstrufinie}
Assume that there exists $K \in \{1, \ldots, p\}$ such that $H_0$ and $H_1$ satisfy
\begin{equation}
\label{hyp_matrices}
\langle H_1 \varphi_1, \varphi_K\rangle =0 
\quad
\text{ and }
\quad 
a_K^1:=  \langle \left[  H_1, \left[H_0, H_1 \right] \right] \varphi_1, \varphi_K \rangle  \neq 0. 
\end{equation}
Then the system \eqref{ODE} is not $W^{-1, \infty}$-STLC: there exists $C>0$ such that for all $a \in (0, \frac{|a^1_K|}{8})$, there exists $T^*> 0$ such that for all $T \in (0,T^*)$, there exists $\eta > 0$ such that for all $u \in L^1(0,T)$ with $\|u_1\|_{L^{\infty}(0,T)} < \eta$, the solution $X$ of \eqref{ODE} with initial data $\varphi_1$ satisfies
%\begin{equation}
%\label{hyp_mvt_contraint}
%\langle X(T), \varphi_j \rangle =0 
%\quad 
%\text{ for some } j \in \{2, \ldots p \} \text{ such that}
%\quad
%\langle H_1 \varphi_1, \varphi_j \rangle \neq 0,
%\end{equation}
%then
%\begin{multline}
%\label{conclu_dim_fin}
%\Im \langle X(T), \varphi_1 e^{-i \lambda T}\rangle 
%\\
%\left\{
%    \begin{array}{lll}
%        &\ioe (\eps_1-a^1_1) \|u_1\|^2_{L^2(0,T)} +\eps_2 \| \left( X - X_{\eq} \right) (T) \| + C\| \left( X - X_{\eq} \right) (T) \| ^2
% , &\text{ if } a_1^1> 0, \\
%        &\soe -(\eps_1+a^1_1) \|u_1\|^2_{L^2(0,T)} -\eps_2 \| \left( X - X_{\eq} \right) (T) \| -C\| \left( X - X_{\eq} \right) (T) \| ^2,  &\text{ if } a^1_1< 0.
%    \end{array}
%\right.
%\end{multline}
\begin{equation}
\label{conclu_dim_fin}
- \sign(a^1_K)
\Im \langle X(T), \varphi_K e^{-i \lambda_1 T}\rangle 
\soe 
%\frac{(-a+ \sign(a^1_K) a^1_K)}{8} 
a
\|u_1\|^2_{L^2(0,T)} 
-C
\| ( X - X_{1} )(T) \|^2.
\end{equation}
\end{thm} 
\begin{rem}
This result is contained in \cite[Theorem 3]{BM18} where the authors gave Lie brackets along which the quadratic order induces a drift, denying STLC of the nonlinear system. However, the proof is still done here in the easier setting of Schrödinger ODEs as it provides a guideline to study obstructions for the Schrödinger PDE \eqref{Schrodinger}. Moreover, a new tool, introduced in \cite{BLM21}, called the Magnus expansion in the interaction picture, could also be used to prove this result. However, this tool seems to not be suitable for an infinite-dimensional framework, thus is not discussed here. 
\end{rem}
\noindent \cref{obstrufinie} relies on the power series expansion of the solution of \eqref{ODE} around the trajectory $(X_{1}, u \equiv 0)$.
%\begin{equation*}
%X_{\eq}(t)= \varphi_1 e^{-i \lambda_1 t}, \quad u_{\eq}(t) = 0.
%\end{equation*}
The first and second-order terms of the expansion are respectively solutions of,
\begin{align}
\label{system_order1_aux}
i X_L' &= H_0 X_L - u(t) H_1 X_{1},
\quad 
\ \text{with }
\quad 
X_L(0)=0, 
\\
\label{system_order2_aux}
\text{and } \quad
i X_Q' &= H_0 X_Q - u(t) H_1 X_L,
\quad 
\ \text{ with }
\quad 
X_Q(0)=0.
\end{align}
%\begin{equation}
%\begin{array}{rll}
%\label{system_order1_and_order2_aux}
%i \dot{X_L} &= H_0 X_L - u(t) H_1 X_{\eq}
%\quad 
%&\text{ with }
%\quad 
%X_L(0)=0, 
%\\
%\label{system_order2_aux}
%\text{and } \quad
%i \dot{X_Q} &= H_0 X_Q - u(t) H_1 X_L
%\quad 
%&\text{ with }
%\quad 
%X_Q(0)=0.
%\end{array}
%\end{equation}
We prove that the solution of the linearized system misses the $K$-th direction and that the second-order term is, up to some negligible terms, a quadratic form, with a coercivity quantified by the $H^{-1}$-norm of the control. We also prove estimates on the cubic remainder of the expansion to show that the quadratic term allows to deny STLC for the full nonlinear equation. 

\subsection{Error estimates on the expansion of the solution}
The goal of this section is to give a rigorous meaning to the following expansion 
\begin{equation}
\label{expansion_dim_finie}
X \approx X_{1} + X_L + X_Q.
\end{equation}
In this section, we consider the following asymptotic. 
\begin{defi}
\label{def_O}
Given two scalar quantities $A(T,u)$ and $B(T,u)$, we write $A(T,u)=\O \left( B(T,u) \right)$ if there exists $C, T^*>0$ such that for any $T \in (0, T^*)$, there exists $\eta > 0$ such that for all $u \in L^1(0,T)$ with $\|u_1\|_{L^{\infty}(0,T)} <  \eta$, we have $|A(T,u)| \ioe C |B(T,u)|.$ 
\end{defi}
Thus, the notation $\O$ refers to the convergence $\| u_1\|_{L^{\infty}(0,T)} \rightarrow 0$ but with controls $u$ in $L^1(0,T)$ to ensure the well-posedness of \eqref{ODE} and all the equations considered. Moreover, this convergence holds uniformly with respect to the final time on a small time interval $[0, T^*]$. 

\subsubsection{Error estimates for the auxiliary system}
\label{subsubsec:estim_aux_fin}
We need sharp error estimates to prove that the cubic remainder of the expansion \eqref{expansion_dim_finie} can be neglected in front of the drift $\| u_1\|_{L^2}^2$. Therefore, classical error estimates involving the $L^2$-norm of the control $u$ are not enough. One can compute estimates involving rather the $L^2$-norm of the time-primitive $u_1$ of the control by introducing the new state
\begin{equation}
\label{def_syst_aux_fin}
\tild{X}(t) := e^{-i H_1 u_1(t)} X(t).
\end{equation}
Such an idea was introduced in \cite{C06} and later used in \cite{BM14} for the Schrödinger equation. This strategy can also be found in \cite{BM18} to study the quadratic behavior of differential systems or in \cite{BLM21} to give refined error estimates for various expansions of scalar-input affine control systems. The new state $\tild{X}$ solves the following ODE, called the auxiliary system,
\begin{equation}
\label{ODEaux}
\tild{X}'(t) 
= -i e^{-iH_1 u_1(t)} H_0 e^{i H_1 u_1(t)} \tild{X}(t) 
= -i \sum \limits_{k=0}^{+\infty} \frac{ \left(-i u_1(t)\right)^k}{k!} \ad_{H_1}^k(H_0) \tild{X}(t).
\end{equation}
\begin{rem}
Taking $\| \cdot \|$ a sub-multiplicative norm on $\mathcal{M}_p(\R)$, one gets
\begin{equation*}
\forall t \in [0,T],  
\quad 
\left\| 
\frac{
u_1(t)^k
}
{k!} 
\ad_{H_1}^k(H_0)
\right\|
\ioe 
\left\| H_0 \right\| 
\frac{
\left(
2 |u_1(t)| 
\left\| H_1 \right\| 
\right)^k
}
{k!},
\end{equation*}
proving the absolute convergence of the series given in \eqref{ODEaux}. The second equality of \eqref{ODEaux} is obtained showing that the matrices on both sides satisfy the following Cauchy problem 
%\begin{equation*}
%\dot{x}(t)
%= 
%i u(t)
%\left[
%x(t),
%H_1
%\right]
%\quad 
%\text{ with }
%\quad
%x(0)=0.  
%\end{equation*}
\begin{equation*}
M'(t)
= 
-i u(t)
\ad_{H_1}(M(t))
\quad 
\text{ with }
\quad
M(0)=H_0,
%\quad
%\text{ and thus }
%x(t) = -i \exp( -i u_1(t) \ad_{H_1} )H_0.
\end{equation*}
giving that $e^{-iH_1 u_1(t)} H_0 e^{i H_1 u_1(t)} = \exp(-i u_1(t) \ad_{H_1})H_0$. 
\end{rem}
\noindent By the Duhamel formula, the solution of the auxiliary system satisfies
\begin{equation}
\label{expr:aux_fin}
\tild{X}(t)
=
X_{1}(t) 
-i
 \int_0^t 
e^{-i H_0 (t-\tau)} 
%\left[ u_1(t) + \ad_{H_1}(H_0) + \frac{i u_1(t)^2}{2} \ad_{H_1}^2(H_0) + O( | u_1^3(t) | ) \right] 
\sum \limits_{k=1}^{+\infty} \frac{ (-i u_1(\tau))^k}{k!} \ad_{H_1}^k(H_0)
\tild{X}(\tau) d\tau.
\end{equation}
The first-order term $\tild{X}_L$ and the second-order term $\tild{X}_Q$ of the expansion of $\tild{X}$ around the trajectory $(X_{1}, u \equiv 0)$ are given by, for all $t \in [0, T]$,  
\begin{align}
\label{expr:X_1,L}
\tild{X}_L(t) &=  -\int_0^t e^{-i H_0 (t-\tau)}  u_1(\tau) \ad^1_{H_1}(H_0) X_{1}(\tau) d\tau, \\
\label{expr:X_1,Q}
\tild{X}_Q(t) &=  \int_0^t e^{-i H_0 (t-\tau)} \left( -u_1(\tau) \ad^1_{H_1}(H_0) \tild{X}_L(\tau) +  \frac{i u_1(\tau)^2}{2} \ad_{H_1}^2(H_0) X_{1}(\tau) \right) d\tau.
\end{align}
In the following, for a function $f : [0,T] \rightarrow \C^p$, we denote by $\| f \|_{\infty} := \sup_{t \in [0,T]} \| f(t) \|$. 
\begin{prop}
The following error estimates on the expansion of the auxiliary system \eqref{ODEaux} hold 
\begin{align}
\label{rl_fin}
\| \tild{X} -X_{1} \|_{\infty} &= \O \left( \| u_1 \|_{L^1(0,T)} \right),
\\
\label{rq_fin}
\| \tild{X} -X_{1}-\tild{X}_L \|_{\infty} &= \O \left( \| u_1 \|^2_{L^2(0,T)} \right),
\\
\label{rc_fin}
\| \tild{X} -X_{1}-\tild{X}_L-\tild{X}_Q \|_{\infty} &= \O \left( \| u_1 \|^3_{L^3(0,T)} \right).
\end{align}
\end{prop}
\begin{proof}
\emph{Bound on the solution.} Taking the scalar product of \eqref{ODE} with $X$ and then the imaginary part of the corresponding equality, one gets
\begin{equation*}
\frac{d}{dt} \| X(t) \|^2= 0 
\quad
\text{ and thus }
\quad
\forall t \in [0,T], 
\quad
\|
X(t)
\|
= 
\|
\varphi_1
\|.
\end{equation*}
Thus, as $H_1$ is symmetric and $u$ is real-valued, 
%the matrices $e^{\pm i H_1 u_1(t)}$ are isometries, 
the definition \eqref{def_syst_aux_fin} of the auxiliary system entails that, for all $t \in [0,T]$, $\| \tild{X}(t)\|=\| \varphi_1\|$. 
%
%As $H_1$ is symmetric and $u$ is real-valued, the matrices $e^{\pm i H_1 u_1(t)}$ are bounded uniformly in time. Hence, equation \eqref{ODEaux} and then Grönwall's Lemma lead to,  
%\begin{equation*}
%\| \dot{X_1}(t) \| \ioe \| H_0 \| \| X_1(t) \|
%\quad 
%\text{ and thus }
%\quad 
%\| X_1(t) \| \ioe \| \varphi_1 \| \exp( T \| H_0 \|),
%\quad \forall t \in [0,T]. 
%\end{equation*}
%Using the definition \eqref{def_syst_aux_fin} of the auxiliary state, the same bound holds for $X$ the solution of the initial system. 

\smallskip \noindent \emph{Proof of \eqref{rl_fin}.} With \eqref{expr:aux_fin} and the boundness of $\tild{X}$, one directly has,
\begin{equation*}
\| ( \tild{X}-X_{1})(t) \| 
=
\left\| 
\int_0^t
e^{-i H_0(t-\tau)}
\O( | u_1(\tau) | )
\tild{X}(\tau) 
d\tau 
\right\| 
= 
\O
\left(
\| u_1\|_{L^1(0,T)}
 \right).
\end{equation*}

\smallskip \noindent \emph{Proof of \eqref{rq_fin}.} Putting together \eqref{expr:aux_fin} and \eqref{expr:X_1,L}, one gets
\begin{multline*}
(
\tild{X} -X_{1}-\tild{X}_L
)
(t)
=
-
\int_0^t 
e^{-i H_0 (t-\tau)}  
u_1(\tau) 
\ad^1_{H_1}(H_0) 
(
\tild{X}-X_{1}
)(\tau)
d\tau
\\
-i
\int_0^t 
e^{-i H_0 (t-\tau)}  
\sum \limits_{k=2}^{+\infty} \frac{ (-i u_1(\tau))^k}{k!} \ad_{H_1}^k(H_0)
\tild{X}(\tau) 
d\tau
\\
=
\O
\left(
\| u_1 \|_{L^1} \| \tild{X}-X_{1} \|_{\infty} 
+
\| u_1\|_{L^2}^2 \| \tild{X}\|_{\infty}
\right).
\end{multline*} 
Thus, the boundness of $\tild{X}$ and estimate \eqref{rl_fin} of $\tild{X}-X_{1}$ lead to \eqref{rq_fin}.

\smallskip \noindent \emph{Proof of \eqref{rc_fin}.} The proof is similar using equations \eqref{expr:aux_fin}-\eqref{expr:X_1,Q} and estimates \eqref{rl_fin}-\eqref{rq_fin}.
\end{proof}

\subsubsection{Error estimates for the initial system}
\begin{prop}
\label{estim_error_with_bc}
The following error estimates on the expansion of the initial system \eqref{ODE} hold 
\begin{align*}
\left\| 
\left(
X -X_{1}
\right)
(T) 
\right\| 
&= 
\O 
\left( 
\| u_1\|_{L^1} 
+
| u_1(T)|
\right),
\\
%\label{rq_fin_init}
\left\| 
\left(
X -X_{1}-X_{L}
\right)
(T) 
\right\| 
&= 
\O 
\left( 
\| u_1 \|^2_{L^2}
%+ \|u_1\|_{L^1} lu_1(T)|
+ | u_1(T) |^2
\right),
\\
%\label{rc_fin_init}
\left\| 
\left( 
X -X_{1}-X_{L}-X_{Q}
\right)
(T)
\right\| 
&= 
\O \left(
\| u_1 \|^3_{L^3} 
%+
%\|u_1\|^2_{L^2} | u_1(T)|
%+
%\| u_1\|_{L^1} l u_1(T)|^2
+
|u_1(T)|^3
\right).
\end{align*}
\end{prop}

\begin{proof}
As the proofs of the three estimates are similar, we only prove here the second one. Identifying the expansions of the initial and auxiliary systems in \eqref{def_syst_aux_fin}, one gets,
\begin{equation}
\label{XL} 
X_L(T) = \tild{X}_L(T) + i u_1(T) H_1 X_{1}(T). 
%X_Q(t) &= X_{1,Q}(t) + i u_1(t)H_1 X_{1,L}(t) - \frac{u_1(t)^2}{2} H_1^2X_{\eq}(t).
%\label{XQ}
\end{equation}
Thus, the relations \eqref{def_syst_aux_fin} and \eqref{XL} lead to, 
\begin{multline*}
\left\|
\left(X - X_{1} -X_L\right)(T)
\right\| 
\ioe 
\| 
e^{i H_1 u_1(T)}
(
\tild{X} - X_{1} - \tild{X}_L 
) (T)
\| 
\\+ 
\|
( e^{iH_1 u_1(T)} -1 ) \tild{X}_L(T)
\| 
+
\| 
( e^{iH_1 u_1(T)} -1 -iH_1 u_1(T)) X_{1}(T) 
\|.  
%&= O \left( \|u_1\|^2_{L^2} + |u_1(T)|^2 + |u_1(T)| \|X_{1,L}(T)\| \right) \\
%&=O \left( \|u_1\|^2_{L^2} + |u_1(T)|^2 \right),
\end{multline*}
The first term of the right-hand side is estimated by $\| u_1\|^2_{L^2}$ using estimate \eqref{rq_fin}.  Doing an expansion of the exponential, the last term is estimated by $|u_1(T)|^2$ and the second term is estimated by $\| \tild{X}_L(T) \| |u_1(T)|$. Finally, looking at \eqref{expr:X_1,L}, the norm of $\tild{X}_L(T)$ is estimated by $\| u_1\|_{L^1}$.
%
%
%\smallskip \noindent The proof of the linear and cubic remainder are similarly using the relations \eqref{def_syst_aux_fin}, \eqref{XL}, \eqref{XQ} and using estimate \eqref{rc_fin} on the auxiliary system. 
\end{proof}
To conclude the estimate of the cubic remainder, it only remains to show that the boundary term $u_1(T)$ can be neglected. This can be done by noticing that such a term appears in the dynamic of the linearized system. 
\begin{prop}
\label{lem:bc_fin}
If $H_0$ and $H_1$ satisfy $a^1_K \neq 0$, then the solution $X$ of \eqref{ODE} with initial data $\varphi_1$ satisfies
\begin{equation}
\label{estim_u1_fin}
|u_1(T)| 
= 
\O
\left( 
\sqrt{T} \|u_1\|_{L^2(0,T)}
+
\| (  X-X_{1}) (T) \|
\right).
\end{equation}
\end{prop}

\begin{proof}
Notice that the assumption $a^1_K \neq 0$ on $H_0$ and $H_1$ entails the existence of $j \in \{1, \ldots, p \}$ such that $\langle H_1 \varphi_1, \varphi_j\rangle \neq 0$ (indeed, if $H_1 \varphi_1=0$ then $[H_1, [H_0, H_1]]\varphi_1=0$). Set such $j$. Solving \eqref{system_order1_aux}, straightforward computations give,  
 \begin{equation*}
 \langle 
 \left(
 X-X_{1}
 \right)
 (T)
 , 
 \varphi_j e^{-i \lambda_1 T}\rangle 
 = 
 i \langle H_1 \varphi_1, \varphi_j\rangle  \int_0^T u(t) e^{i (\lambda_j - \lambda_1)(t-T)}dt
 +
 \O \left( 
\left\| (X - X_{1} -X_L)(T) \right\| 
 \right).
 \end{equation*}
Yet, if $j\neq 1$, doing one integration by parts and using Cauchy-Schwarz inequality, one gets
\begin{equation*}
\int_0^T u(t) e^{i (\lambda_j - \lambda_1)(t-T)}dt 
%&= u_1(T) - i (\lambda_j- \lambda_1)\int_0^T u_1(t) e^{i (\lambda_j - \lambda_1)(t-T)}dt, \\
=u_1(T) + \O(\sqrt{T} \|u_1\|_{L^2(0,T)}).
\end{equation*}
This also holds if $j=1$. Therefore, using \cref{estim_error_with_bc} to estimate $(X - X_{1} -X_L)(T)$, one gets 
\begin{equation*}
|u_1(T)| 
= \O
\left(
\sqrt{T} \|u_1\|_{L^2} 
+
\| u_1 \|^2_{L^2}
%+ 
%\| u_1\|_{L^1}
%| u_1(T) |
+ 
| u_1(T) |^2
+
\| (X-X_1) (T) \|
\right).
\end{equation*}
By definition of the notation $\O$, those estimates are done under the asymptotic $\| u_1 \|_{L^{\infty}}$ going to zero. Thus, such estimate entails \eqref{estim_u1_fin}. 
\end{proof}
%
%
%
%From \cref{estim_error_with_bc} and \ref{lem:bc_fin}, one deduces the following estimates. 
\begin{cor}
\label{estim_error}
If $H_0$ and $H_1$ satisfy $a^1_K \neq 0$, the following error estimate holds on the expansion of the solution of \eqref{ODE},
\begin{equation}
%\label{rq_fin_init}
%\left\| 
%\left(
%X -X_{\eq}-X_{L} 
%\right)
%(T)
%\right\| 
%&= 
%\O \left( 
%\| u_1 \|^2_{L^2(0,T)}
% \right),
%\\
%\label{rc_fin_init}
\label{proof2}
\left\| 
\left(
X -X_{1}-X_{L}-X_{Q}
\right)
(T) 
\right\|
= 
\O 
\left(
\| u_1 \|^3_{L^3} 
+
%+
%\| u_1\|^2_{L^2} 
%\| (X-X_{\eq}) (T) \|
%+
%\| u_1\|_{L^1} 
%\| (  X-X_{\eq}) (T) \|^2
\| (  X-X_1) (T) \|^3
\right).
\end{equation}
\end{cor}

\subsection{Study of the quadratic term}
The goal of this section is to prove that the second-order term of the expansion \eqref{expansion_dim_finie} has a drift. First, notice that the solution of the linearized system \eqref{system_order1_aux} is given by, 
\begin{equation}
\label{XLexp}
X_L(t) = 
i \sum \limits_{j=1}^p 
\left(
\langle H_1 \varphi_1, \varphi_j\rangle 
\int_0^t u(\tau) e^{i(\lambda_j- \lambda_1) \tau} d\tau
\right)
\varphi_j e^{-i \lambda_j t}
, \quad t \in (0,T).
\end{equation}
Therefore, plugging this expression into the second-order system \eqref{system_order2_aux}, one has 
\begin{multline*}
\langle X_Q(T), \varphi_K e^{-i \lambda_1 T}\rangle  
%= i \int_0^T u(t) \langle  H_1 X_L(t), \varphi_1\rangle  e^{i\lambda_1 t} dt 
\\
= - 
\sum \limits_{j=1}^p
\langle H_1 \varphi_1, \varphi_j\rangle 
\langle H_1 \varphi_K, \varphi_j\rangle
 \int_0^T u(t) 
 \int_0^t u(\tau) 
 e^{i 
 \left(
\lambda_j (\tau-t)
+
\lambda_K(t-T)
+
\lambda_1(T- \tau)
 \right)
 } 
 d\tau dt.
\end{multline*}
%  
%\begin{align*}
%\int_0^T u(t) &\int_0^{t} u(\tau) e^{i (\lambda_1 - \lambda_j)( t - \tau)} d\tau \\ &= u_1(T) \int_0^T u(\tau) e^{i (\lambda_1 - \lambda_j)( T - \tau)} d\tau  - \int_0^T u_1(t) \left( u(t) + i(\lambda_1-\lambda_j) \int_0^t u(\tau) e^{i (\lambda_1 - \lambda_j)( t - \tau)}  d\tau \right) dt \\
%&=\frac{u_1(T)^2}{2} + i(\lambda_1- \lambda_j)u_1(T) \int_0^T u_1(\tau) e^{i (\lambda_1 - \lambda_j)( T - \tau)}  d\tau - i (\lambda_1-\lambda_j) \int_0^T u_1(t)^2 dt \\ &+ (\lambda_1-\lambda_j)^2 \int_0^T u_1(t) \int_0^{\tau} u_1(\tau) e^{i (\lambda_1 - \lambda_j)( t - \tau)} dt d\tau. 
%\end{align*}
%
%\noindent So, summing over $j$, and using that $(\lambda_j- \lambda_1) \langle H_1 \varphi_1, \varphi_j\rangle  = \langle  [H_0, H_1] \varphi_1, \varphi_j\rangle $, we finally get, 
The coercivity quantified by the $H^{-1}$-norm of the control $u$ is revealed by integrations by parts as follows,
\begin{multline*}
\Im
\langle X_Q(T), \varphi_K e^{-i \lambda_1 T}\rangle 
 =  
Q(u_1)
\\
+
u_1(T) 
\sum \limits_{j=1}^p
\langle [H_0, H_1] \varphi_1, \varphi_j \rangle
\langle H_1 \varphi_K, \varphi_j \rangle
\int_0^T
u_1( \tau)
\cos[(\lambda_j- \lambda_1)(\tau-T)]
d\tau
,
\end{multline*}
where $Q$ is a quadratic form on $L^2(0,T)$ given by, 
\begin{multline}
\label{def_Q}
Q(u_1)
:=
- \frac{a^1_K}{2} 
\int_0^T 
u_1(t)^2
\cos[(\lambda_K-\lambda_1)(t-T)]
dt
+
\sum \limits_{j=1}^p
\langle [H_0, H_1] \varphi_1, \varphi_j \rangle
\langle [H_0, H_1] \varphi_j, \varphi_K \rangle
\\
\times
\int_0^T
u_1(t)
\int_0^t
u_1( \tau)
\sin 
\left(
\lambda_j (\tau-t)
+
\lambda_K(t-T)
+
\lambda_1(T- \tau)
\right)
d\tau dt.
\end{multline}
%
%\begin{multline*}
%\langle X_Q(T), \varphi_K e^{-i \lambda_1 T}\rangle 
% =  
%-i \frac{a^1_K}{2}
%\int_0^T 
%u_1(t)^2
%e^{i (\lambda_K- \lambda_1)(t-T)}
%dt
%\\+
%\sum \limits_{j=1}^p
%\langle [H_0, H_1] \varphi_1, \varphi_j \rangle
%\langle [H_0, H_1] \varphi_j, \varphi_K \rangle
%\int_0^T
%u_1(t)
%\int_0^t
%u_1( \tau)
% e^{i 
% \left(
%\lambda_j (\tau-t)
%+
%\lambda_K(t-T)
%+
%\lambda_1(T- \tau)
%\right)
% } 
% d\tau dt
%\\
% -\frac{u_1(T)^2}{2} 
%\langle H_1^2 \varphi_1, \varphi_K \rangle
%+
%i
%u_1(T) 
%\sum \limits_{j=1}^p
%\langle [H_0, H_1] \varphi_1, \varphi_j \rangle
%\langle H_1 \varphi_K, \varphi_j \rangle
%\int_0^T
%u_1( \tau)
%e^{i (\lambda_j- \lambda_1)(\tau-T)} 
%d\tau
% .
%\end{multline*}
%The last term are neglected using Cauchy-Schwarz inequality and the boundary term $u_1(T)$ is neglected thanks to \cref{lem:bc_fin} to get,
Using \cref{lem:bc_fin} to neglect the boundary term $u_1(T)$ and Cauchy-Schwarz inequality, one gets, 
%\begin{equation*}
%\Im  \langle X_Q(T), \varphi_1 e^{-i \lambda_1 T}\rangle 
%=  - a_1^1 \int_0^T u_1^2(t) dt + \O \left( |u_1(T)|^2 + T \|u_1\|^2_{L^2(0,T)} \right).
%\end{equation*}
%
%Besides, using \cref{lem:bc_fin} to estimate the boundary term $u_1(T)$, one finally gets,
\begin{equation}
\label{proof3}
\Im  
\langle 
X_Q(T), 
\varphi_K 
e^{-i \lambda_1 T}
\rangle 
=  
Q(u_1)
+  
\O
\left( 
T \|u_1\|_{L^2}^2
+ 
 \| ( X - X_1) (T) \|^2 
\right).
\end{equation}
Moreover, as stated below, the quadratic form $Q$ has a coercivity quantified by the $L^2$-norm of $u_1$.
\begin{lem}
\label{coercivity_finie}
There exists $T^*>0$ such that for every $T \in (0, T^*)$ and $u_1 \in L^2(0,T)$, 
\begin{equation*}
- 
\sign(a^1_K) 
Q(u_1) 
\soe 
\frac{|a^1_K|}{8}
\int_0^T u_1(t)^2 dt. 
\end{equation*}
%\begin{equation*} Q(u_1)  \left\{
%    \begin{array}{ll}
%        &\ioe - \frac{a^1_K}{8} \int_0^T u_1(t)^2 dt, \quad \text{ if } a^1_K> 0, \\
%        &\soe - \frac{a^1_K}{8} \int_0^T u_1(t)^2 dt, \quad \text{ if } a^1_K< 0.
%    \end{array}
%\right. \end{equation*}
\end{lem} 
\begin{proof}
Assume $a^1_K> 0$ (the case $a^1_K<0$ is similar). Let $T \in (0, T^*)$ with $T^*$ to be determined and $u_1 \in L^2(0,T)$. 

\noindent \emph{Upper negative bound on the first term of $Q$.} If $T^* \ioe \frac{\pi}{3 (\lambda_K - \lambda_1)}$, as cosine is decreasing on $[0, \frac{\pi}{3}]$, one gets,
\begin{equation}
\label{estim_Qn_1_fin}
-\frac{a^1_K}{2} \int_0^T u_1(t)^2 \cos[(\lambda_K-\lambda_1)(t-T)] dt \ioe -\frac{a^1_K}{4} \int_0^T u_1(t)^2 dt.
\end{equation}
\noindent \emph{Estimate of the second term of $Q$.} If we denote by 
$$C_K^1 :=  
\sum \limits_{j=1}^p
\left|
\langle [H_0, H_1] \varphi_1, \varphi_j \rangle
\langle [H_0, H_1] \varphi_j, \varphi_K \rangle
\right|,
$$ 
%one has, for all $(t, \tau) \in [0, T]^2$,
%\begin{equation*}
%\left|
%\sum \limits_{j=1}^p
%\langle [H_0, H_1] \varphi_1, \varphi_j \rangle
%\langle [H_0, H_1] \varphi_j, \varphi_K \rangle
%\sin 
%\left(
%\lambda_j (\tau-t)
%+
%\lambda_K(t-T)
%+
%\lambda_1(T- \tau)
%\right)
%d\tau dt
%\right|
%\ioe C_K^1.
%\end{equation*}
by Cauchy-Schwarz inequality, the absolute value of the second term in \eqref{def_Q} is bounded by, 
\begin{equation}
\label{estim_Qn_2_fin}
%\left|
%\sum \limits_{j=1}^p
%\langle [H_0, H_1] \varphi_1, \varphi_j \rangle
%\langle [H_0, H_1] \varphi_j, \varphi_K \rangle
%\int_0^T
%u_1(t)
%\int_0^t
%u_1( \tau)
%\sin 
%\left(
%\lambda_j (\tau-t)
%+
%\lambda_K(t-T)
%+
%\lambda_1(T- \tau)
%\right)
%d\tau dt
%\right|
%\\
C_K^1T \int_0^T u_1(t)^2 dt
\ioe 
\frac{a^1_K}{8} \int_0^T u_1(t)^2 dt,
\end{equation}
if we choose $T^* \ioe \frac{a^1_K}{8C_K^1}$.

\noindent \emph{Conclusion.} Putting together \eqref{estim_Qn_1_fin}, \eqref{estim_Qn_2_fin} with the definition of $Q$ \eqref{def_Q}, the proof follows with 
%\begin{equation*}  T^*= \left\{
%    \begin{array}{lll}
%        \frac{|a^1_1|}{8C_1^1}, \quad &\text{ if } K=1, \\
%        \min \left( \frac{|a^1_K|}{8 C_K^1} ; \frac{\pi}{3(\lambda_K-\lambda_1)} \right), \quad &\text{ if } K \geqslant 2.
%    \end{array}
%\right. \end{equation*}
\begin{equation*}
T^* 
:=
\frac{|a^1_1|}{8C_1^1},
\quad
\text{if }
K=1
\quad 
\text{ and }
\quad
T^*
:=
\min \left( \frac{|a^1_K|}{8 C_K^1} ; \frac{\pi}{3(\lambda_K-\lambda_1)} \right),
\quad
\text{if }
K \neq 1.
\end{equation*}
\end{proof}
%
%%%Calculs dans le cas K=1
%The coercivity quantified by the $H^{-1}$-norm of the control is revealed by integrations by parts,
%\begin{multline*}
%\Im \langle X_Q(T), \varphi_K e^{-i \lambda_1 T}\rangle 
% =  
%- \sum \limits_{j=2}^p \langle  [H_0, H_1] \varphi_1, \varphi_j\rangle ^2 \int_0^T u_1(t)  \int_0^t u_1(\tau) \sin( (\lambda_j - \lambda_1)(\tau- t)) d\tau dt   
%\\
%-a_1^1 \int_0^T u_1^2(t) dt
%+ u_1(T) \sum \limits_{j=2}^p \langle  [H_0, H_1] \varphi_1, \varphi_j\rangle  \langle H_1 \varphi_1, \varphi_j\rangle  \int_0^T u_1( \tau)  \cos( (\lambda_j- \lambda_1) (\tau -T)) d\tau. 
%\end{multline*}
%The first and last integrals are neglected using Cauchy-Schwarz inequality and the boundary term $u_1(T)$ is neglected thanks to \cref{lem:bc_fin} to get,
%\begin{equation*}
%\Im  \langle X_Q(T), \varphi_1 e^{-i \lambda_1 T}\rangle 
%=  - a_1^1 \int_0^T u_1^2(t) dt + \O \left( |u_1(T)|^2 + T \|u_1\|^2_{L^2(0,T)} \right).
%\end{equation*}
%
%Besides, using \cref{lem:bc_fin} to estimate the boundary term $u_1(T)$, one finally gets,
%\begin{equation}
%\label{proof3}
%\Im  
%\langle 
%X_Q(T), 
%\varphi_K 
%e^{-i \lambda_1 T}
%\rangle 
%=  
%- a_K^1 \int_0^T u_1^2(t) dt 
%+  
%\O
%\left( 
%T \|u_1\|_{L^2}^2
%+ 
%\| u_1\|_{L^1} \| ( X - X_{\eq}) (T) \| 
%\right).
%\end{equation}
%%Fin des calculs du cas K=1
%

\subsection{Proof of the first quadratic obstruction on Schrödinger ODEs}
\label{proof_obstru_dim_finie}
First, using the explicit form of $X_L$ given in \eqref{XLexp}, when $\langle H_1 \varphi_1, \varphi_K \rangle=0$, one notices that the $K$-th direction is lost at the linear level in the sense that 
\begin{equation*}
\langle X_L(T), \varphi_K \rangle =0.
\end{equation*}
Thus, using the work on the quadratic term \eqref{proof3} and the estimate on the cubic remainder \eqref{proof2}, one has
\begin{align*}
\Im 
\langle X(T), 
\varphi_K 
e^{-i \lambda_1 T} 
\rangle 
&=
\Im
\langle X_Q(T), 
\varphi_K 
e^{-i \lambda_1 T} 
\rangle 
+
\O
\left(
\| 
(
X
-X_1
-X_L
-X_Q
)
(T)
\|
\right)
\\
&=
Q(u_1)
+ 
\O
\left( 
T \|u_1\|_{L^2}^2
+
\| u_1\|^3_{L^3}
+
\| (  X-X_1) (T) \|^2
\right).  
\end{align*}
Expanding the notation $\O$ (see \cref{def_O}), this means that there exists $C, T_1 >0$ such that for all $T \in (0, T_1)$, there exists $\eta_1>0$ such that for all $u \in L^1(0,T)$ with $\| u_1 \|_{L^{\infty}(0,T)} < \eta_1$, one has, 
\begin{equation*}
\left|
\Im 
\langle X(T), 
\varphi_K 
e^{-i \lambda_1 T} 
\rangle 
-
Q(u_1)
\right|
\ioe 
C
\left( 
T \|u_1\|_{L^2}^2
+
\| u_1\|^3_{L^3}
+
\| (  X-X_1) (T) \|^2
\right).
\end{equation*}
Let $a \in (0, \frac{|a^1_K|}{8})$, $T^*:= \min(T_1, \frac{|a^1_K|}{16C}-\frac{a}{2C})$, $T \in (0,T^*)$ and $\eta:=\min(\eta_1, \frac{|a^1_K|}{16C}-\frac{a}{2C})$. Then, for all  $u \in L^{1}(0,T)$ with $\| u_1 \|_{L^{\infty}(0,T)} < \eta$, one has,
\begin{equation*}
\left|
\Im 
\langle X(T), 
\varphi_K 
e^{-i \lambda_1 T} 
\rangle 
-Q(u_1)
\right|
\ioe 
\left(
\frac{|a^1_K|}{8} 
-
a
\right)
\| u_1 \|^2_{L^2}
+
C 
\| ( X - X_1) (T) \| ^2
.
\end{equation*}
Together with the coercivity of $Q$ given in \cref{coercivity_finie}, this concludes the proof of \cref{obstrufinie}. 

\section{Well-posedness of the Schrödinger equation}
\label{wellposedness}
In this section, we recall the result about the well-posedness of the following Cauchy problem, given in \cite[Proposition 2]{BL10} and later extended in \cite{B21},
\begin{equation}  \label{Cauchypb} \left\{
    \begin{array}{ll}
        i \partial_t \psi(t,x) = - \partial^2_x \psi(t,x) -u(t)\mu(x)\psi(t,x) -f(t,x),  \quad &(t,x) \in (0,T) \times (0,1),\\
        \psi(t,0) = \psi(t,1)=0, \quad &t \in (0,T), \\
        \psi(0,x)=\psi_0(x), \quad &x \in (0,1).
    \end{array}
\right.  \end{equation}
 
\begin{prop} \label{WP}
Let $T> 0$, $\mu \in H^3((0,1), \mathbb{R})$, $\psi_0 \in H^3_{(0)}(0,1)$, $f \in L^2((0,T), H^3 \cap H^1_0)$ and $u \in L^2((0,T), \mathbb{R})$. There exists a unique weak solution of \eqref{Cauchypb} i.e a function $\psi \in C^0([0,T], H^3_{(0)})$ such that the following equality holds in $H^3_{(0)}(0,1)$ for every $t \in [0,T]$, 
$$\psi(t)= e^{-iAt} \psi_0 + i \int_0^t e^{-iA(t - \tau)} \Big( u(\tau) \mu \psi(\tau) + f(\tau) \Big) d\tau.$$
Moreover, for every $R> 0$, there exists $C=C(T, \mu, R) > 0$ such that if $\|u\|_{L^2(0,T)} <  R$ then this weak solution satisfies 
$$\| \psi \|_{ C^0([0,T], H^3_{(0)}(0,1))} \ioe C \left( \| \psi_0\|_{H^3_{(0)}(0,1)} + \|f\|_{L^2((0,T), H^3 \cap H^1_0(0,1))} \right).$$
If $f \equiv 0$, then
$$\|\psi(t) \|_{L^2(0,1)} = \| \psi_0 \|_{L^2(0,1)}, \quad \forall t \in [0,T].$$ 
\end{prop}
Therefore, from now on, we will always work with controls $u$ at least in $L^2(0,T)$ and with $\mu$ at least in $H^3(0,1)$ to ensure the well-posedness of all the equations considered. We will also always take $\varphi_1$ as the initial condition of \eqref{Schrodinger}. Besides, as highlighted in \cite{B21}, the solution of \eqref{Cauchypb} and more precisely its regularity, relies strongly on the control and the dipolar moment. However, for sake of simplicity, in the following, we will not mention this dependency. 

\section{Error estimates on the expansion of the solution}
\label{expansion_solution}

\subsection{Formal expansion of the solution}
The main tool to prove \cref{obstruction_n} is the power series expansion of the solution of the Schrödinger equation \eqref{Schrodinger} around the ground state. By \cref{WP}, for $u \in L^2(0,T)$ and $\mu \in H^3(0,1)$, one may consider, 
\begin{itemize}
\item the first-order term $\Psi$, in $C^0([0,T], H^3_{(0)})$, solution of the linearized equation given by,
\begin{equation}  \left\{
    \begin{array}{ll}
        i \partial_t \Psi = - \partial^2_x \Psi -u(t)\mu(x)\psi_1, \\
        \Psi(t,0) = \Psi(t,1)=0,\\
        \Psi(0,x)=0, 
    \end{array}
\right. \label{order1} 
\end{equation}
which can be explicitly computed as, 
\begin{equation}
\label{order1explicit}
\Psi(t)=i 
\sum \limits_{j=1}^{+\infty}
\left(
 \langle \mu \varphi_1, \varphi_j\rangle  \int_0^t u(\tau) e^{i (\lambda_j-\lambda_1)\tau} d\tau
 \right)
  \psi_j(t), \quad t \in [0,T], 
\end{equation}
\item and the second-order term, $\xi$ in $C^0([0,T], H^3_{(0)})$, which is the solution of, 
\begin{equation}  \left\{
    \begin{array}{ll}
        i \partial_t \xi = - \partial^2_x \xi -u(t)\mu(x)\Psi, \\
        \xi(t,0) = \xi(t,1)=0,  \\
        \xi(0,x)=0.
    \end{array}
\right. \label{order2} \end{equation}
\end{itemize}
The proof of \cref{obstruction_n} is in two steps.
\begin{itemize}
\item First, we understand in which way the following expansion holds rigorously:
\begin{equation*}
\psi \approx \psi_1 + \Psi + \xi.
\end{equation*}
%More precisely, one needs to compute a sharp error estimate of such expansion to state that along the lost direction, the behavior of the solution, at the final time $T$, is given by the quadratic term $\langle \xi(T), \varphi_K \rangle$. 
\item Then, we prove that the quadratic term entails a drift, quantified by the $H^{-n}$-norm of the control, preventing $H^{2n-3}$-STLC when $n\soe2$ and $W^{-1, \infty}$-STLC when $n=1$, for the full nonlinear system. 
\end{itemize}
First, we specify the smallness assumption on the controls that we use in all the following.

\begin{defi}
\label{def_O_bis}
Given two scalar quantities $A(T,u)$ and $B(T,u)$, we write $A(T,u)=\O \left( B(T,u) \right)$ if there exists $C, T^*>0$ such that for any $T \in (0, T^*)$, there exists $\eta > 0$ such that for all $u \in L^2(0,T)$ with $\|u_1\|_{L^{\infty}(0,T)} <  \eta$, we have $|A(T,u)| \ioe C |B(T,u)|.$ 
\end{defi}
As in \cref{def_O}, this notation refers to the convergence $\| u_1\|_{L^{\infty}(0,T)} \rightarrow 0$ uniformly with respect to the final time. However, to ensure the well-posedness of the Schrödinger equation \eqref{Schrodinger}, one asks for the controls $u$ to be at least in $L^2(0,T)$ and not just in $L^1(0,T)$ as in finite dimension. 

\subsection{The auxiliary system}
As discussed in \cref{subsubsec:estim_aux_fin}, to prove that the behavior of the nonlinear solution is driven by the quadratic term of the expansion, one needs to compute sharp error estimates, not quantified with respect to the control $u$ but with respect to the time-primitive $u_1$ of the control. This is done again by the means of an auxiliary system:
if $\psi$ is a solution of \eqref{Schrodinger}, we consider a new state $\wt{\psi}$ given by, 
\begin{equation}
\label{defaux}
\wt{\psi}(t,x):=\psi(t,x) e^{-i u_1(t) \mu(x)}, \quad (t,x) \in [0,T] \times [0,1],
\end{equation}
which is a weak solution of 
\begin{equation}  
 \left\{ 
    \begin{array}{ll}
        i \partial_t \wt{\psi} = - \partial^2_x \wt{\psi} -iu_1(t) ( 2 \mu'(x) \partial_x \wt{\psi} + \mu''(x) \wt{\psi} ) +u_1(t)^2 \mu'(x)^2 \wt{\psi}, \\
        \wt{\psi}(t,0) = \wt{\psi}(t,1)=0,\\
        \wt{\psi}(0,x)=\varphi_1(x). 
    \end{array}
\right. \label{aux} \end{equation}
The well-posedness of this equation is stated below as in \cite[Proposition 2]{BM14}.
\begin{prop}
\label{wp_aux}
Let $T> 0$, $\mu \in H^3((0,1), \mathbb{R})$, $u_1 \in H^1((0,T), \mathbb{R})$ with $u_1(0)=0$. There exists a unique weak solution $\wt{\psi} \in C^0([0,T], H^3 \cap H^1_0(0,1))$ of \eqref{aux}. Moreover, for every $R> 0$, there exists $C=C(T, \mu, R)> 0$ such that, if $\|u_1\|_{H^1(0,T)} < R$, then this weak solution satisfies 
$$\| \wt{\psi} \|_{ C^0([0,T], H^3 \cap H^1_0(0,1))} \ioe C.$$
\end{prop}

\begin{rem} 
\label{rem:wp_aux}
Because of the term $\partial_x \tild{\psi}$ in \eqref{aux}, it does not seem possible to use a fixed-point theorem to prove directly the well-posedness of \eqref{aux} when $u_1 \in L^2((0,T), \R)$. Thus, up to now, the solution of \eqref{aux} is only understood through its link \eqref{defaux} with the Schrödinger equation \eqref{Schrodinger}. Thus, the regularity and the bound on $\tild{\psi}$ stated in \cref{wp_aux} follow from \cref{WP} and therefore hold under assumptions on $u_1'$ and not just on $u_1$.  

Furthermore, when $u_1$ is in $H^1((0,T), \R)$, $\tild{\psi}$ is a weak solution of \eqref{aux} in the sense that the following equality holds in $H^1_0(0,1)$ for every $t \in [0,T]$, 
\begin{equation*}
\tild{\psi}(t)
=
\psi_1(t)
-
\int_0^t
e^{-iA(t-\tau)}
\left(
u_1(\tau)
\left(
2 \mu' \partial_x 
+
\mu''
\right)
\tild{\psi}(\tau)
+
i
u_1(\tau)^2 \mu'^2 \tild{\psi}(\tau)
\right)
d\tau.
\end{equation*}
Notice that the right-hand side of the equality is indeed in $H^1_0(0,1)$ thanks to the smoothing effect stated below in \cref{estimfine}, which was highlighted in \cite{BL10} and later used in \cite{B21}.  

Moreover, when $u_1$ is in $C^1([0,T], \R)$, $\psi$ and thus $\tild{\psi}$ are in $C^1( [0,T], L^2(0,1))$  and the first equation of \eqref{aux} is satisfied in $L^2(0,1)$ at every time. 
\end{rem}
\begin{lem}
\label{estimfine}
There exists a nondecreasing function $C: \R_+ \rightarrow \R_+^*$ such that for all $T \soe 0$ and $f$ in $L^2((0,T),H^1(0,1))$, the function $G : t \mapsto \int_0^t e^{-iA(t-\tau)} f(\tau) d\tau$ is in $C^0([0,T],H^1_0(0,1))$ with 
$$\|G\|_{C^0([0,T],H^1_0(0,1))} \ioe C(T) \|f\|_{L^2((0,T),H^1(0,1))}.$$
\end{lem}
Now, we want to study the expansion around the ground state of the solution $\tild{\psi}$ of the auxiliary system \eqref{aux}. 

\medskip \noindent \emph{First-order term.}
Linearizing \eqref{defaux}, the first-order term $\wt{\Psi}$ of the expansion of $\tild{\psi}$ is given by
\begin{equation}
\label{defauxlin}
 \wt{\Psi}(t,x)=\Psi(t,x) - i u_1(t) \mu(x) \psi_1(t,x),
 \end{equation}
where $\Psi$ is the solution of \eqref{order1}. Thus,  $\wt{\Psi}$ is in $C^0([0,T], H^3 \cap H^1_0)$ and is a weak solution of 
\begin{equation}   \left\{ 
    \begin{array}{ll}
        i \partial_t \wt{\Psi} = - \partial^2_x \wt{\Psi} -iu_1(t) \left( 2 \mu' \partial_x \psi_1 + \mu'' \psi_1 \right) , \\
        \wt{\Psi}(t,0) = \wt{\Psi}(t,1)=0,\\
        \wt{\Psi}(0,x)=0. 
    \end{array}
\right. \label{order1aux} \end{equation}
%Estimates on $\tild{\Psi}$ can be obtained using the following smoothing effect, first highlighted in \cite{BL10} and later used in \cite{B21}.
%
%
%\begin{prop}
%For all $T^*>0$ and $\mu \in H^3( (0,1), \R)$, there exists $C>0$ such that for all $T \in (0, T^*)$ and $u_1 \in H^1(0,1)$ with $u_1(0)=0$, 
%\begin{equation}
%\label{estim_Psi_tild}
%\| \tild{\Psi} \|_{L^{\infty}( (0,T), H^1_0)} \ioe C \| u_1\|_{L^2(0,T)}.
%\end{equation}
%\end{prop}
%\begin{proof}
%Solving \eqref{order1aux}, $\tild{\Psi}$ satifies, for all $t \in [0,T]$, 
%\begin{equation*}
%\tild{\Psi}(t) 
%= 
%-
%\int_0^t
%e^{-iA (t- \tau)} 
%u_1(\tau) 
%\left( 2 \mu' \varphi_1' + \mu'' \varphi_1 \right)
%e^{-i \lambda_1 \tau}
%d\tau. 
%\end{equation*}
%As the function $u_1 \left( 2 \mu' \varphi_1' + \mu'' \varphi_1 \right) e^{-i \lambda_1 \tau}$ is in $L^2( (0,T), H^1(0,1))$, \cref{estimfine} entails that 
%\begin{equation*}
%\| \tild{\Psi}(t) \|_{H^1_0(0,1)}
%\ioe 
%C(T^*) 
%\| u_1 \left( 2 \mu' \varphi_1' + \mu'' \varphi_1 \right) \|_{L^2( (0,T), H^1(0,1))}
%\ioe 
%C(T^*, \mu) \| u_1\|_{L^2(0,T)}. 
%\end{equation*}
%\end{proof}
%
\emph{Second-order term.}
Doing an expansion of order 2 of \eqref{defaux}, the second-order term $\wt{\xi}$ is given by,  
\begin{equation}
\label{defauxqua}
\wt{\xi}(t,x)= \xi(t,x) - i u_1(t) \mu(x) \wt{\Psi}(t,x) + \frac{u_1(t)^2}{2} \mu(x)^2 \psi_1(t,x),
\end{equation}
where $\xi$ is the solution of \eqref{order2}. Notice that $\wt{\xi}$ is in $C^0([0,T], H^3 \cap H^1_0)$ and is a weak solution of, 
\begin{equation}   \left\{ 
    \begin{array}{ll}
        i \partial_t \wt{\xi} = - \partial^2_x \wt{\xi} -iu_1(t) ( 2 \mu' \partial_x \wt{\Psi} + \mu'' \wt{\Psi} ) + u_1(t)^2 \mu'^2 \psi_1 , \\
        \wt{\xi}(t,0) = \wt{\xi}(t,1)=0,\\
        \wt{\xi}(0,x)=0. 
    \end{array}
\right. \label{order2aux} \end{equation}
%\red{Mettre une estimation sur $\tild{\xi}$ aussi ???}

\begin{prop}
\label{prop:estim_Psi_xi_tild}
The first and second-order terms of the expansion of $\tild{\psi}$ satisfy
\begin{equation}
\label{estim_Psi_xi_tild}
\| \tild{\Psi} \|_{L^{\infty}( (0,T), H^1_{0}(0,1))}
=
\O
\left(
\| u_1\|_{L^2}
\right)
\quad 
\text{ and }
\quad 
\| \tild{\xi} \|_{L^{\infty}( (0,T), L^2(0,1))}
=
\O
\left(
\| u_1\|^2_{L^2}
\right)
.
\end{equation}
\end{prop}
\begin{proof}
First, solving \eqref{order1aux}, the following equality holds in $H^1_0(0,1)$,
\begin{equation*}
\tild{\Psi}(t) 
=
-
\int_0^t
e^{-iA(t -\tau)}
u_1(\tau)
\left(
2 \mu' \partial_x+ \mu''
\right)
\psi_1(\tau)
d\tau, 
\quad
t \in [0,T].
\end{equation*}
As the function 
$
\tau 
\mapsto 
u_1(\tau)
\left(
2 \mu' \partial_x+ \mu''
\right)
\psi_1(\tau)
$
is in $L^2((0,T), H^1(0,1))$, by \cref{estimfine}, one gets the existence of $C=C(T)>0$ such that 
\begin{equation*}
\|  
\tild{\Psi}
\|_{C^0([0,T],H^1_0)} 
\ioe C(T)
\|
u_1 
\left(
2 \mu' \partial_x+ \mu''
\right)
\psi_1
\|_{L^2((0,T),H^1)}
\ioe 
C(T) \| \mu \|_{H^3(0,1)} 
\|u_1 \|_{L^2(0,T)},
\end{equation*}
also using the algebra structure of $H^3(0,1)$. As the constant is nondecreasing with respect to the final time, one gets the first estimate of \eqref{estim_Psi_xi_tild}. Moreover, solving \eqref{order2aux}, one gets 
\begin{equation*}
\tild{\xi}(t)
=
-
\int_0^t 
e^{-iA(t -\tau)}
\left(
u_1(\tau)
\left(
2 \mu' \partial_x + \mu'' 
\right)
\tild{\Psi}(\tau)
+
i
u_1(\tau)^2 
\mu'^2 
\psi_1(\tau)
\right)
d\tau, 
\quad 
t \in [0, T]. 
\end{equation*}
Therefore, using the triangular inequality together with the fact that for all $s \in \R$, $e^{iAs}$ is an isometry from $L^2(0,1)$ to $L^2(0,1)$, one gets, 
\begin{align*}
\| \tild{\xi}(t) \|_{L^2(0,1)}
&\ioe
\int_0^t
\left(
| u_1(\tau)|
\| 
\left(
2 \mu' \partial_x + \mu'' 
\right)
\tild{\Psi}(\tau)
\|_{L^2(0,1)}
+
| u_1(\tau)|^2
\| 
\mu'^2 \psi_1(\tau)
\|_{L^2(0,1)}
\right)
d\tau
\\
&\ioe
\| \mu\|_{H^3}
\| u_1\|_{L^1(0,T)}
\| \tild{\Psi} \|_{L^{\infty}( (0,T), H^1_0)}
+
\| \mu'^2 \|_{L^{\infty}}
\| u_1\|^2_{L^2(0,T)}
\| \psi_1\|_{L^{\infty}( (0,T), L^2)}.
\end{align*}
Thus, the estimate on $\tild{\Psi}$ allows to conclude the proof of the second estimate of \eqref{estim_Psi_xi_tild}. 
\end{proof}

\subsection{Energy estimates on the auxiliary system } 
One of our goals is to prove that the first obstruction can occur assuming only that $\| u_1\|_{L^{\infty}}$ is small. Hence, we seek to prove estimates on $\tild{\psi}$ not assuming the boundness of $u_1$ in $H^1(0,T)$ as required in \cref{wp_aux}. 

\begin{lem}
\label{conservation_L2_norm}
For every $T>0$, $\mu \in H^3((0,1), \R)$ and $u_1 \in H^1( (0, T), \R)$, the solution $\tild{\psi}$ of the auxiliary system \eqref{aux} satisfies
\begin{equation}
\label{eq:conservation_aux}
\forall t \in [0,T], 
\quad
\| \tild{\psi}(t) \|_{L^2(0,1)}
=
1
.
\end{equation}
\end{lem}
This first result follows directly from the definition of the auxiliary system \eqref{defaux} and the conservation of the $L^2$-norm of the solution of the Schrödinger equation given in \cref{WP}. 
\begin{prop}
For every $\mu \in H^3((0,1), \R)$, the solution $\tild{\psi}$ of the auxiliary system \eqref{aux} satisfies
\begin{equation}
\label{eq:estimate_psi-psi1}
\| 
\tild{\psi} - \psi_1
\|_{L^{\infty}((0,T), L^2(0,1))}
=
\O
\left( 
\| u_1\|_{L^2(0,T)}
\right)
.
\end{equation}
\end{prop}

\begin{proof}
The estimate is computed first for regular controls so that the equation \eqref{aux} holds in $L^2(0,1)$ at every time (see \cref{rem:wp_aux}) and then deduced by density. Denote by $\tild{R}:= \tild{\psi}- \psi_1$. Looking at \eqref{aux}, $\tild{R}$ is the solution of 
\begin{equation}
\label{eq_R_tild}
i \partial_t \tild{R} 
=
- \partial^2_x \tild{R} 
-iu_1(t) \left( 2 \mu' \partial_x + \mu'' \right) \tild{R} 
-iu_1(t) \left( 2 \mu' \partial_x  + \mu''  \right)  \psi_1
+u_1(t)^2 \mu'^2 \tild{\psi},
\end{equation} 
with Dirichlet boundary conditions and initial condition $\tild{R}(0,\cdot)=0$. Let $t \in [0,T]$.  The proof consists in taking the $L^2$-scalar product of \eqref{eq_R_tild} with $\tild{R}$, integrating over $[0,t] $ and taking the imaginary part. 
First, notice that 
\begin{equation}
\label{eq:conservation_1}
\Im
\left(
i 
\int_0^t 
\langle
\partial_t \tild{R}(\tau),
\tild{R}(\tau)
\rangle 
d\tau 
\right)
=
\frac{1}{2}
\int_0^t 
\frac{d}{dt}
\| \tild{R}(\tau)\|^2_{L^2(0,1)}
d\tau
=
\frac{1}{2}
\| \tild{R}(t)\|^2_{L^2(0,1)}
.
\end{equation}
Moreover, as for every $\tau \in [0,T]$, $\tild{R}(\tau)$ is in $H^1_0(0,1)$, one integration by parts gives,
\begin{equation}
\label{eq:conservation_2}
\Im
\left(
-
\int_0^t 
\langle 
\partial_x^2 \tild{R}(\tau),
\tild{R}(\tau) 
\rangle
d\tau 
\right)
=
\Im
\left(
\int_0^t 
\|
\partial_x \tild{R}(\tau)
\|_{L^2(0,1)}^2
d\tau 
\right)
=
0
.
\end{equation}
Besides, as the operator $2 \mu' \partial_x + \mu''$ is skew-Hermitian on $H^1_0(0,1)$ and $u_1$ is real-valued, 
\begin{equation}
\label{eq:conservation_3}
\Im
\left(
i 
\int_0^t
u_1(\tau)
\langle
\left( 2 \mu' \partial_x  + \mu'' \right) \tild{R}(\tau)
,
\tild{R}(\tau)
\rangle
d\tau
\right)
=0.
\end{equation}
%\begin{equation*}
%\Im
%\left(
%\int_0^t
%\langle
%i \partial_t \tild{R}(\tau) 
%+ 
%\partial^2_x \tild{R}(\tau) 
%+
%iu_1(\tau) ( 2 \mu' \partial_x  + \mu'' ) \tild{R}(\tau)
%,
%\tild{R}(\tau)
%\rangle
%d\tau
%\right)
%\\
%= 
%\frac{1}{2} 
%\| \tild{R}(t)\|^2_{L^2(0,1)}. 
%\end{equation*}
Moreover, using Young and Cauchy-Schwarz inequalities, for every control such that $\| u_1\|_{L^{\infty}} \ioe 1$, 
\begin{multline}
\label{eq:conservation_4}
\left|
\int_0^t
\langle
-i
u_1(\tau)
\left( 2 \mu' \partial_x  + \mu'' \right) \psi_1(\tau)
+
u_1(\tau)^2 
\mu'^2
\tild{\psi}(\tau)
,
\tild{R}(\tau)
\rangle
d\tau
\right|
\\
\ioe 
\frac{1}{2}
\| u_1\|^2_{L^2}
+
\frac{1}{2}
\left(
\| 
( 2 \mu' \partial_x  + \mu'') \psi_1
\|_{L^{\infty}(L^2)}^2
+
\| 
\mu'^2
\tild{\psi}
\|_{L^{\infty}(L^2)}^2
\right)
\int_0^t
\| \tild{R}(\tau)\|^2_{L^2}
d\tau
.
\end{multline}
Thus, the equation \eqref{eq_R_tild} together with estimates \eqref{eq:conservation_1}-\eqref{eq:conservation_4} and estimate \eqref{eq:conservation_aux} on $\tild{\psi}$, one gets the existence of $C=C(\mu)>0$ such that, for all $u_1$ such that $\| u_1\|_{L^{\infty}} \ioe 1$, 
\begin{equation*}
\| \tild{R}(t)\|^2_{L^2(0,1)}
\ioe 
C  \| u_1\|_{L^2}^2
+
C
\int_0^t 
\| \tild{R}(\tau)\|^2_{L^2(0,1)}
d\tau.
\end{equation*}
Therefore, Gronwall's Lemma leads to \eqref{eq:estimate_psi-psi1}, as the definition of $\O$ means that we work in the asymptotic $\|u_1\|_{L^{\infty}}$ small (see \cref{def_O_bis}).  
\end{proof}

\subsection{Error estimates for the auxiliary system}
\begin{prop}
\label{estimates_aux}
For every $\mu \in H^3((0,1), \R)$ and $p \in \N^*$, the following scalar error estimates hold
\begin{align}
\label{eq:estimate_psi-psi1-Psi}
\langle  
e^{i u_1(T) \mu}
(\widetilde{\psi} - \psi_1 -\widetilde{\Psi})(T)
, 
\varphi_p
\rangle 
&= 
\O
\left(
\|u_1\|^2_{L^2(0,T)}
\right) 
,
\\
\label{eq:estimate_psi-psi1-Psi-xi}
\langle  
e^{i u_1(T) \mu}
(\widetilde{\psi} - \psi_1 -\widetilde{\Psi} - \widetilde{\xi})(T)
, 
\varphi_p
\rangle 
&= 
\O
\left(
\|u_1\|^3_{L^2(0,T)}
\right). 
\end{align}
\end{prop}

\begin{proof} 
\emph{Proof of \eqref{eq:estimate_psi-psi1-Psi}.} Solving \eqref{aux} and \eqref{order1aux}, the following equality holds in $H^1_0(0,1)$, 
\begin{equation}
\label{eq:expr_rq2}
(\widetilde{\psi} - \psi_1 - \widetilde{\Psi})(T) 
= 
- \int_0^T
e^{-iA(T-t)} 
\left(
u_1(t) 
\left(
2 \mu' \partial_x  + \mu''
\right) 
(\tild{\psi}-\psi_1)(t)
+ 
i 
u_1(t)^2 
\mu'^2 
\widetilde{\psi}(t)
\right)
dt.
\end{equation}
Thus, 
$\langle  
e^{i u_1(T) \mu}
(\widetilde{\psi} - \psi_1 -\widetilde{\Psi})(T)
, 
\varphi_p
\rangle 
=
I_1 
+
I_2, 
$
where 
\begin{align*}
I_1 
&:=
- \int_0^T
u_1(t)
\langle
e^{-iA(T-t)}
(
2 \mu' \partial_x
+ \mu''
)
(
\tild{\psi}
-\psi_1
)(t)
,
e^{-i u_1(T)\mu}
\varphi_p
\rangle
dt,
\\
I_2 
&:= 
-i 
\int_0^T u_1(t)^2 
\langle 
e^{-iA(T-t)}
\mu'^2 
\tild{\psi}(t)
,
e^{-i u_1(T)\mu}
\varphi_p
\rangle
dt.
\end{align*}
Using Cauchy-Schwarz's inequality, that $e^{iAs} : L^2 \rightarrow L^2$ is an isometry and estimate \eqref{eq:conservation_aux} on $\tild{\psi}$, one gets
\begin{equation*}
|
I_2
|
\ioe
\int_0^T
u_1(t)^2
\| e^{-iA(T-t)} \mu'^2 \tild{\psi}(t) \|_{L^2(0,1)}
\| e^{-i u_1(T) \mu} \varphi_p \|_{L^2(0,1)}
dt
\ioe 
\| \mu'\|^2_{L^{\infty}}
\| u_1\|^2_{L^2}
.
\end{equation*}
Moreover, using that the operator $2 \mu' \partial_x + \mu''$ is skew-Hermitian on $H^1_0$ and Cauchy-Schwarz inequality, one gets
\begin{multline*}
| 
I_1 
|
=
\left|
\int_0^T 
u_1(t)
\langle 
(\tild{\psi} - \psi_1)(t)
,
(2 \mu' \partial_x+ \mu'') 
e^{iA(T-t)}
e^{-i u_1(T) \mu} \varphi_p
\rangle
dt 
\right|
\\
\ioe 
\| u_1\|_{L^1}
\| \tild{\psi} - \psi_1 \|_{L^{\infty}( (0,T), L^2)}
\| 
(2 \mu' \partial_x+ \mu'') 
e^{iA(T-t)}
e^{-i u_1(T) \mu} \varphi_p
\|_{L^{\infty}( (0,T), L^2)}
=\O
\left( 
\sqrt{T} \| u_1\|^2_{L^2}
\right),
\end{multline*}
using estimate \eqref{eq:estimate_psi-psi1} on $\tild{\psi}-\psi$ and because, for all time $t$, 
\begin{multline}
\label{estim_avec_phase}
\| 
(2 \mu' \partial_x+ \mu'') 
e^{iA(T-t)}
e^{-i u_1(T) \mu} \varphi_p
\|_{L^2}
\ioe
3 \| \mu\|_{H^3}
\| 
e^{-i u_1(T) \mu} \varphi_p
\|_{H^1_0}
\\
=
3 \| \mu\|_{H^3}
\| 
e^{-iu_1(T)\mu}
\left(
\varphi_p'
-iu_1(T) \mu' \varphi_p
\right)
\|_{L^2}
=\O(1),
\end{multline}
using that $e^{iAs}: H^1_0(0,1) \rightarrow H^1_0(0,1)$ is an isometry and recalling that we work with $\|u_1\|_{L^{\infty}}$ small and thus bounded by definition of $\O$.

\medskip \noindent \emph{Proof of \eqref{eq:estimate_psi-psi1-Psi-xi}.}
Solving \eqref{aux}, \eqref{order1aux} and \eqref{order2aux}, one gets 
$
\langle
e^{i u_1(T) \mu}
(\widetilde{\psi} - \psi_1 -\widetilde{\Psi} - \widetilde{\xi})(T)
,
\varphi_p
\rangle 
=
J_1
+
J_2,
$
where
\begin{align*}
J_1
&:=
- 
\int_0^T  
u_1(t) 
\langle 
e^{-iA(T-t)}
\left(
2\mu' \partial_x + \mu'' 
\right)
(\widetilde{\psi} - \psi_1 - \widetilde{\Psi})(t)
, 
e^{-i u_1(T) \mu}\varphi_p 
\rangle
dt, 
\\
J_2
&:=
- i
\int_0^T  
u_1(t)^2 
\langle 
e^{-iA(T-t)}
\mu'^2 (\widetilde{\psi} - \psi_1)(t), 
e^{-i u_1(T) \mu}
\varphi_p
\rangle
dt. 
\end{align*}
%
%
%\begin{multline}
%\label{expr:rc_aux}
%\langle
%e^{i u_1(T) \mu}
%(\widetilde{\psi} - \psi_1 -\widetilde{\Psi} - \widetilde{\xi})(T)
%,
%\varphi_p
%\rangle 
%= 
%- 
%\int_0^T  
%u_1(t) 
%\langle 
%e^{-iA(T-t)}
%\left(
%2\mu' \partial_x + \mu'' 
%\right)
%(\widetilde{\psi} - \psi_1 - \widetilde{\Psi})(t)
%, 
%e^{-i u_1(T) \mu}\varphi_p 
%\rangle
%dt 
%\\ 
%- i
%\int_0^T  
%u_1(t)^2 
%\langle 
%e^{-iA(T-t)}
%\mu'^2 (\widetilde{\psi} - \psi_1)(t), 
%e^{-i u_1(T) \mu}
%\varphi_p
%\rangle
%dt 
%.
%\end{multline}
%
%\noindent \emph{Estimate of the second integral.}
As before, using \eqref{eq:estimate_psi-psi1} to estimate $\tild{\psi} - \psi_1 $, one gets,
\begin{equation}
\label{eq:reste_cub_1}
%\left|
%\int_0^T  
%u_1(t)^2 
%\langle 
%\mu'^2 (\widetilde{\psi} - \psi_1)(t), 
%\varphi_K\rangle
%e^{i \lambda_K (t-T)} 
%dt 
%\right|
%\ioe 
|
J_2
|
\ioe 
\| \mu'\|^2_{L^{\infty}}
\| u_1\|^2_{L^2}
\| 
\tild{\psi} - \psi_1 
\|_{L^{\infty}( (0,T), L^2)}
=
\O
\left(
\| u_1\|^3_{L^2}
\right)
.
\end{equation}
%
%
%\noindent \emph{Estimate of the first integral.} 
Moreover, as the operator $2 \mu' \partial_x + \mu''$ is skew-Hermitian on $H^1_0$, $J_1$ is given by 
\begin{equation*}
J_1=
\int_0^T  
u_1(t) 
\langle 
(\widetilde{\psi} - \psi_1 - \widetilde{\Psi})(t)
, 
(2 \mu' \partial_x + \mu'')
[
e^{iA(T-t)}
e^{-iu_1(T)\mu}\varphi_p
]
\rangle
dt
. 
\end{equation*}
Recalling the computation of $\widetilde{\psi} - \psi_1 - \widetilde{\Psi}$ given in \eqref{eq:expr_rq2}, one can write  $J_1=J_{1,1}+J_{1,2}$ with 
\begin{align*}
J_{1,1}&:=
-
\int_0^T
u_1(t)
\int_0^t
u_1(\tau)
\times
\\
&
\langle 
e^{-iA(t-\tau)}
\left(
2 \mu' \partial_x  + \mu''
\right) 
(\tild{\psi}-\psi_1)(\tau)
,
(2 \mu' \partial_x + \mu'')
[
e^{iA(T-t)}
e^{-iu_1(T)\mu}
\varphi_p
]
\rangle
d\tau dt 
,
\\
J_{1,2}&:=
-
i 
\int_0^T 
u_1(t)
\int_0^t
u_1(\tau)^2
\langle 
e^{-iA(t-\tau)} 
\mu'^2 
\widetilde{\psi}(\tau)
, 
(2 \mu' \partial_x + \mu'')
[
e^{iA(T-t)}
e^{-iu_1(T)\mu}\varphi_p
]
\rangle
d\tau dt 
.
\end{align*}
%
%\noindent \emph{Estimate on $J_2$.} 
Using Cauchy-Schwarz inequality in $L^2(0,1)$, estimate \eqref{eq:conservation_aux} on $\tild{\psi}$ and estimate \eqref{estim_avec_phase}, one gets 
\begin{equation}
\label{eq:reste_cub_2}
\left|
J_{1,2} 
\right|
=
\O\left(
\| u_1 \|_{L^1(0,T)} 
\| u_1 \|_{L^2(0,T)}^2
\| \tild{\psi} \|_{L^{\infty}((0,T), L^2(0,1))} 
\right)
=
\O
\left(
\sqrt{T} 
\| u_1\|^3_{L^2}
\right)
.
\end{equation}
%\noindent \emph{Estimate on $I_1$.}
Besides, for all $(t, \tau) \in [0,T]$, using Cauchy-Schwarz inequality and estimate \eqref{estim_avec_phase}, one gets
\begin{multline*}
\left|
u_1(t)
u_1(\tau)
\langle 
e^{-iA(t-\tau)}
\left(
2 \mu' \partial_x  + \mu''
\right) 
(\tild{\psi}-\psi_1)(\tau)
, 
(2 \mu' \partial_x + \mu'')
[
e^{iA(T-t)}
e^{-iu_1(T)\mu}
\varphi_p
]
\rangle
\right|
\\
=\O
\left(
| u_1(t) |
| u_1(\tau)|
\| (\tild{\psi}-\psi_1)(\tau) \|_{H_0^1(0,1)}
\right)
\in 
L^1( (0,T) \times (0,T)),
\end{multline*}
as $u_1$ is in $L^{2}(0,T)$ and $(\tild{\psi}-\psi_1)$ is in $C^0( [0,T], H^3 \cap H^1_0(0,1))$. Thus, one can apply  Fubini's theorem to write $J_{1,1}$ as 
\begin{multline*}
J_{1,1}
=
-
\int_{\tau = 0}^T 
u_1(\tau)
\langle 
\left(
2 \mu' \partial_x  + \mu''
\right) 
(\tild{\psi}-\psi_1)(\tau)
,
F(\tau)
\rangle
d\tau
\\
\text{ with }
\quad 
F(\tau)
:= 
\int_{t = \tau}^T
e^{iA(t-\tau)}
u_1(t)
(2 \mu' \partial_x + \mu'')
[
e^{iA(T-t)}
e^{-iu_1(T)\mu}\varphi_p
]
dt
.
\end{multline*}
%Yet, for every $t \in [0,T]$, as $e^{-iA(T-t)}$ is an isometry from $H^2_{(0)}$ to $H^2_{(0)}$, 
%\begin{equation*}
%\| 
%e^{iA(T-t)}
%e^{-iu_1(T)\mu}
%\varphi_p
%\|_{H^2_{(0)}}
%=
%\| 
%e^{-iu_1(T)\mu}\varphi_p 
%\|_{H^2_{(0)}}
%\ioe 
%C
%\| 
%e^{-i u_1(T) \mu}
%\left(
%\varphi_p''
%-
%2i u_1(T) \mu' \varphi_p'
%-
%\left(
%u_1(T)^2 \mu'^2
%+i u_1(T) \mu'' 
%\right)
%\varphi_p
%\right)
%\|_{L^2}
%=\O(1),
%\end{equation*}
%as we work with controls small and thus bounded in $W^{-1, \infty}$. 
As in estimate \eqref{estim_avec_phase}, 
for all time $t$, 
$
\| 
e^{iA(T-t)}
e^{-iu_1(T)\mu}
\varphi_p
\|_{H^2_{(0)}}
=
\O
\left(
1
\right)
.
$
Thus, the function 
%the function
%$
%e^{iA(T-t)}
%e^{-iu_1(T)\mu}\varphi_p
%$ 
%is in $H^2_{(0)}$, 
%then
$
u_1( \cdot)
(2 \mu' \partial_x + \mu'')
[
e^{iA(T-\cdot)}
e^{-iu_1(T)\mu}\varphi_p
]
$
is in $L^2( (0,T), H^1)$. So, by \cref{estimfine}, $F$ is in $C^0([0,T], H^1_0)$ with 
\begin{equation}
\label{estim_F}
\|
F
\|_{C^0([0,T],H^1_0)} 
\ioe
C
\|
u_1
(2 \mu' \partial_x + \mu'')
[
e^{iA(T-t)}
e^{-iu_1(T)\mu}\varphi_p
]
\|_{L^2((0,T),H^1)}
=
\O
\left(
\| u_1\|_{L^2}
\right).
\end{equation}
Therefore, using Cauchy-Schwarz inequality and that $2\mu' \partial_x + \mu''$ is skew-Hermitian, \eqref{eq:estimate_psi-psi1} and \eqref{estim_F} lead to 
\begin{equation}
\label{eq:reste_cub_3}
| 
J_{1,1}
|
%=
%\left|
%\int_0^T 
%u_1(\tau)
%\langle 
%(\tild{\psi}-\psi_1)(\tau)
%,
%\left(
%2 \mu' \partial_x  + \mu''
%\right) 
%F(\tau)
%\rangle
%d\tau
%\right|
\ioe
C
\| u_1\|_{L^1}
\| \tild{\psi}-\psi_1\|_{L^{\infty}( (0,T), L^2)}
\| F\|_{L^{\infty}( (0,T), H^1_0)}
=
\O
\left(
\sqrt{T}
\| u_1\|^3_{L^2}
\right)
.
\end{equation}
Estimates \eqref{eq:reste_cub_1}, \eqref{eq:reste_cub_2}  and \eqref{eq:reste_cub_3} conclude the proof of \eqref{eq:estimate_psi-psi1-Psi-xi}.
\end{proof}

\subsection{Error estimates for the Schrödinger equation}
Now, from estimates on the auxiliary system, one can deduce estimates for the Schrödinger equation. 
\begin{prop}
\label{estimates}
For every $\mu \in H^3((0,1), \R)$ and $p \in \N^*$, the following error estimates hold
\begin{align*}
\| \psi- \psi_1\|_{L^{\infty}((0,T), L^2(0,1))} 
&=
\O
\left(
\|u_1\|_{L^2}
+
|u_1(T)|
\right) 
,
\\
\langle  
(\psi - \psi_1 -\Psi)(T)
, 
\varphi_p
\rangle 
&= 
\O
\left(
\|u_1\|^2_{L^2}
+
|u_1(T)|^2
\right) 
,
\\
\langle  
(\psi - \psi_1 -\Psi - \xi)(T)
, 
\varphi_p
\rangle 
&= 
\O
\left(
\|u_1\|^3_{L^2}
+
|u_1(T)|^3
\right). 
\end{align*}
\end{prop}

\begin{proof}
As the proof of the three estimates is similar, we give here only the proof of the last estimate. Using the links \eqref{defaux}, \eqref{defauxlin} and \eqref{defauxqua} between the several systems, we have,
\begin{multline*}
(\psi - \psi_1 - \Psi - \xi)(T) 
=
(
e^{i u_1(T) \mu} - 1 - i u_1(T) + \frac{u_1(T)^2}{2}\mu^2 
)
\psi_1(T) 
+  
(
e^{i u_1(T) \mu} -1 - i u_1(T) \mu
)
\widetilde{\Psi}(T)
\\+
(e^{i u_1(T) \mu} -1)
\widetilde{\xi}(T) 
+ 
e^{i u_1(T) \mu} 
(
\widetilde{\psi} - \psi_1 -\widetilde{\Psi} - \widetilde{\xi}
)(T).
\end{multline*}
We conclude as in the proof of \cref{estim_error_with_bc}: doing an expansion of $e^{i u_1(T) \mu}$, the $p$-th coordinate of the first term (resp. second and third term) can be estimated by $|u_1(T)|^3$ (resp. by $|u_1(T)|^2 \| \tild{\Psi}\|_{L^2}$ and by $|u_1(T)| \| \tild{\xi}\|_{L^2}$). Then, using \eqref{estim_Psi_xi_tild} to estimate $\tild{\Psi}(T)$ and $\tild{\xi}(T)$ and \eqref{eq:estimate_psi-psi1-Psi-xi} to estimate the last term, one concludes the proof. 
\end{proof}
To conclude the error estimate of the cubic remainder, one needs to be able to neglect the boundary term $u_1(T)$. As in  \cref{lem:bc_fin}, it can be done by noticing that such a term arises in the dynamic of the linearized system. The proof is exactly the same with the $L^2(0,1)$-scalar product instead of the $\C^p$-one and with $\mu$ instead of $H_1$ and thus is left to the reader. 
\begin{prop}
\label{estim_u1(T)}
If $\mu$ satisfies \eqref{coeff_quad_non_nul}, then the solution $\psi$ of \eqref{Schrodinger} with initial data $\varphi_1$ satisfies 
\begin{equation}
\label{eq:estim_u1(T)}
u_1(T) = 
\O
\left(
\sqrt{T}
\|u_1\|_{L^2(0,T)}
+
\| (\psi-\psi_1)(T) \|_{L^2(0,1)}
\right).
\end{equation}
\end{prop}

\begin{cor}
For every $\mu \in H^3((0,1), \R)$ satisfying \eqref{coeff_quad_non_nul} and $p \in \N^*$, the following error estimates hold
\begin{align}
\label{estim_rest_quad}
\langle  
(\psi - \psi_1 -\Psi)(T)
, 
\varphi_p
\rangle 
= 
\O
\left(
\|u_1\|^2_{L^2(0,T)}
+
\| (\psi- \psi_1)(T) \|^2_{L^2(0,1)}
\right)
,
\\
\label{estim_rest_cub}
\langle  
(\psi - \psi_1 -\Psi - \xi)(T)
, 
\varphi_p
\rangle 
= 
\O
\left(
\|u_1\|^3_{L^2(0,T)}
+
\| (\psi- \psi_1)(T) \|^3_{L^2(0,1)}
\right). 
\end{align}
\end{cor}

\section{Coercivity of the quadratic term}
\label{section:coercivity}
Most objects defined in this section have a dependency with respect to the final time $T$, the index $K$ of the lost direction and the index $n$ of the obstruction considered. To lighten the notations, we will only mention the dependency with respect to $n$. The goal of this section is to prove that, under (H2)$_{K, n}$, in an appropriate sense, the quadratic term has the following drift
\begin{equation*}
\Im \langle \xi(T), \varphi_K e^{-i \lambda_1 T} \rangle \approx - A^n_K \| u_n \|^2_{L^2(0,T)}.
\end{equation*}
Plugging the explicit form of the first-order term $\Psi$ \eqref{order1explicit} into the second-order system \eqref{order2}, computations lead to 
\begin{equation}
\label{expr_xi}
\langle \xi(T), \varphi_K e^{-i \lambda_1 T} \rangle = \int_0^T u(t) \int_0^t u(\tau) h(t, \tau) d\tau dt,
\end{equation}
where the quadratic kernel $h$ is given by
\begin{equation}
\label{kernel_h}
h(t, \tau)
:=
-
\sum \limits_{j=1}^{+\infty} 
\langle \mu \varphi_K, \varphi_j\rangle  
\langle \mu \varphi_1, \varphi_j\rangle  
e^{
i \left[ 
\lambda_K (t-T)
+ \lambda_j (\tau -t)
+ \lambda_1 (T- \tau)
 \right]}, \quad \forall (t, \tau) \in [0,T]^2.
\end{equation}
By the assumption \eqref{conv_sum} on $\mu$, $h$ is bounded in $C^{2n}( \R^2, \C)$. This regularity is the key to perform integrations by parts to reveal coercive drifts, quantified by any integer negative Sobolev norm. 
\begin{prop}
\label{forme_q}
Let $n \in \N$ and $H \in C^{2n}( \R^2, \C)$. There exists a quadratic form $\wt{Q}_n$ on $\C^{2n}$ such that for all $T>0$ and $u \in L^1(0,T)$,
\begin{multline*}
\int_0^T u(t) \int_0^t u(\tau) H(t, \tau) d\tau dt 
= 
 \sum \limits_{p=1}^n 
\int_0^T 
u_p(t)^2
\left(
\frac{1}{2} \frac{d}{dt}(  \partial_1^{p-1} \partial_2^{p-1} H(t, t)) 
-  \partial_1^{p} \partial_2^{p-1} H(t, t) 
\right)  dt 
\\
+
\int_0^T 
u_n(t) 
\int_0^t 
u_n(\tau) 
\partial_1^n \partial_2^n H(t, \tau) d\tau dt
+
\wt{Q}_n(u_1(T), \ldots, u_n(T), \gamma_0^n, \ldots, \gamma_{n-1}^n),
\end{multline*}
where, for all $p=0 \ldots n-1$, we denote by $\gamma^n_p:=\int_0^T u_n(\tau) \partial_1^p \partial_2^n H(T,\tau) d\tau.$ 
\end{prop}
\begin{proof}
This result is proved by induction on $m \in \{0, \ldots,n\}$. The equality is clear for $m=0$ with the convention that the sum is empty and taking $\tild{Q}_0=0$. Assume it holds for $m \in \{0, \ldots,n-1\}$. Integrations by parts show that 
\begin{multline*}
\int_0^T u_m(t) \int_0^t u_m(\tau) \partial_1^m \partial_2^m H(t, \tau) dt d\tau 
= 
-u_{m+1}(T) \int_0^T u_{m+1}(\tau) \partial_1^m \partial_2^{m+1} H(T, \tau) d\tau
\\
+\frac{u_{m+1}^2(T)}{2} \partial_1^m \partial_2^m H(T,T)  
+ \int_0^T u_{m+1}(t)^2 \left( \frac{1}{2} \frac{d}{dt}(  \partial_1^m \partial_2^m H(t, t)) 
-  \partial_1^{m+1} \partial_2^m H(t, t) \right) dt 
\\+ \int_0^T u_{m+1}(t) \int_0^t u_{m+1}(\tau)  \partial_1^{m+1} \partial_2^{m+1} H(t, \tau) d\tau dt.
\end{multline*}
The conclusion of the induction follows after noticing that, doing once again integrations by parts, 
$$\forall i=0, \ldots, m-1, 
\quad 
\gamma_i^m =u_{m+1}(T)  \partial_1^i \partial_2^m H(T, T) - \gamma_i^{m+1}.$$
\end{proof}
%%%
\noindent With the expression of $h$ \eqref{kernel_h} and the definition of the coefficients $A^p_K$ \eqref{coeff_quad_non_nul}, one computes that, 
\begin{equation*}
\forall p \in \{1, \ldots, n\},
\quad
\frac{1}{2} \frac{d}{dt}(  \partial_1^{p-1} \partial_2^{p-1} h(t, t)) 
-  \partial_1^{p} \partial_2^{p-1} h(t, t) 
= 
-i
A^p_K
e^{i (\lambda_K-\lambda_1)(t-T)}
.
\end{equation*}
%%%
%%%
Therefore, applying \cref{forme_q} to \eqref{expr_xi}, under hypothesis (H2)$_{K, n}$ on  $\mu$, one gets 
\begin{equation}
\label{expr_xi_2}
\Im 
\langle \xi(T), \varphi_K e^{-i \lambda_1 T} \rangle
= 
%- A^n_K \int_0^T u_n(t)^2 \cos( (\lambda_K- \lambda_1)(t-T)) dt
%\\ + 
%\int_0^T u_n(t) \int_0^t u_n(\tau) k_n(t, \tau) d\tau dt
Q_n(u_n)
+ 
\O
\left(
\sum \limits_{p=1}^n 
\left| 
u_p(T)
\right|^2 
+
\left|
\int_0^T u_n(\tau) \partial_1^p \partial_2^n h(T, \tau) d\tau
\right|^2
\right),
\end{equation}
where $Q_n$ is the following quadratic form defined by, for $s$ in $L^2(0,T)$, 
\begin{equation}
\label{Qn}
%\forall s \in L^2(0,T), \
Q_n(s):=-A^n_K \int_0^T s(t)^2 \cos[(\lambda_K-\lambda_1)(t-T)] dt + \int_0^T s(t) \int_0^t s(\tau) k_n(t, \tau) d\tau  dt,
\end{equation}
and where the real quadratic kernel $k_n$ is given by,
\begin{multline}
\label{def_k}
k_n(t, \tau)
%&:= \Im h(t, \tau) 
%\\
:=
(-1)^{n+1}
\sum \limits_{j=1}^{+\infty} 
\left( \lambda_K - \lambda_j \right)^n
\left( \lambda_j - \lambda_1 \right)^n
\langle \mu \varphi_K, \varphi_j\rangle  
\langle \mu \varphi_1, \varphi_j\rangle  
\\
\times
\sin\left(
\lambda_K (t-T)
+ \lambda_j (\tau-t)
+ \lambda_1 (T-\tau)
 \right).
\end{multline}
%Notice that this kernel $k$ and the quadratic form mentioned in the following lemma have a dependency with respect to the final time $T$, the label $K$ of the lost direction at the linearized level and the label $n$ of the obstruction. However, for sake of simplicity, we will omit to recall such dependency.
%%
%%
%%%
The following lemma states the coercivity of the quadratic form $Q_n$. 
\begin{lem}
\label{coercivityn}
Assuming \eqref{conv_sum} and \eqref{coeff_quad_non_nul} on $\mu$, there exists $T^*>0$ such that for every $T \in (0, T^*)$ and $s \in L^2(0,T)$, 
\begin{equation*}
-
\sign(A^n_K)
Q_n(s) 
\soe 
\frac{|A^n_K|}{4} \int_0^T s(t)^2 dt. 
\end{equation*}
%\begin{equation*} Q_n(s)  \left\{
%    \begin{array}{ll}
%        &\ioe - \frac{A^n_K}{4} \int_0^T s(t)^2 dt, \quad \text{ if } A^n_K> 0, \\
%        &\soe - \frac{A^n_K}{4} \int_0^T s(t)^2 dt, \quad \text{ if } A^n_K< 0.
%    \end{array}
%\right. \end{equation*}
\end{lem} 
\begin{proof}
The proof is the same as the proof of \cref{coercivity_finie} with 
\begin{equation*}
T^* 
:=
\frac{|A^n_1|}{4C_1^n},
\quad
\text{if }
K=1
\quad 
\text{ and }
\quad
T^*
:=
\min \left( \frac{|A^n_K|}{4 C_K^n} ; \frac{\pi}{3(\lambda_K-\lambda_1)} \right),
\quad
\text{if }
K \neq 1,
\end{equation*}
where the constant $C_K^n$ bounds the kernel $k_n$ and is given by  
$$
C_K^n 
:=  
\sum \limits_{j=1}^{+\infty} 
\left| 
(\lambda_K - \lambda_j)^n 
(\lambda_j - \lambda_1)^n 
c_j  
\right|,
$$
where $(c_j)_{j \in \N^*}$ is defined in (H2)$_{K, n}$. Notice that $C^n_K$ is finite by \eqref{conv_sum} and non-vanishing by \eqref{coeff_quad_non_nul}. 
\end{proof}
As in \cref{estim_u1(T)}, the boundary terms $(u_p(T))_{p=1, \ldots, n}$ arising in \eqref{expr_xi_2} can be neglected. To that end, we prove first that the linearized system can move in at least $n$ directions.

\begin{lem}
\label{existence_J}
Under (H2)$_{K, n}$, there exists at least $n$ values of $j$ in $\N^*$ such that $c_j \neq 0$, where the sequence $(c_j)_{j \in \N^*}$ is defined in (H2)$_{K, n}$.
%Under (H1)$_K$-(H2)$_{K, n}$, there exists $J \subset \mathbb{N}^* - \{1, K \}$ of cardinal $n$ such that for all $j \in J$, $c_j \neq 0$, where the sequence $(c_j)_{j \in \N^*}$ is defined in (H2)$_{K, n}$. 
%\begin{equation}
%\label{coeff_non_nul}
%\langle 
%\mu \varphi_1, 
%\varphi_j
%\rangle 
%
%\langle 
%\mu \varphi_K, 
%\varphi_j 
%\rangle 
%\neq 0.
%\end{equation}
\end{lem}

\begin{proof}
Assume by contradiction that there exists at most $n-1$ values of $j$ in $\N^*$ such that $c_j \neq 0$. Let denote by $(c_{j_i})_{i=1, \ldots n-1}$ such values. 
Then, for all $p \in \{1, \ldots, n \}$, 
\begin{equation}
\label{coeff_Apk_fin}
A^p_K
=
(-1)^{p-1}
\sum 
\limits_{i=1}^{n-1} 
\left( 
\lambda_{j_i} 
- 
\frac{
\lambda_1+ \lambda_K
}
{
2}
\right)
v_{j_i}^{p-1}
c_{j_i},
\
\text{ where  }
\
v_j
:=
(\lambda_K- \lambda_j)
(\lambda_j - \lambda_1)
.
\end{equation}
Our goal is to prove that in that case, 
\begin{equation}
\label{eq}
-
A^n_K 
=
\sum 
\limits_{p=1}^{n-1}
%(-1)^{p+1}
\sigma_p
A^{n-p}_K
\quad
\text{ where }
\quad
\forall p=1, \ldots, n-1, 
\quad
\sigma_p:=
\sum
\limits_{1 \ioe i_1 < \ldots < i_{p} \ioe n-1}
v_{j_{i_1}} \ldots v_{j_{i_{p}}}
,
\end{equation}
which leads to an absurdity thanks to \eqref{coeff_quad_nul} and \eqref{coeff_quad_non_nul} and thus concludes the proof. Plugging \eqref{coeff_Apk_fin} into the right-hand side of \eqref{eq}, one gets,
\begin{equation*}
\sum 
\limits_{p=1}^{n-1}
\sigma_p
%\Big(
%(-1)^{p+1}
%&\sum
%\limits_{1 \ioe i_1 < \ldots < i_{p} \ioe n-1}
%v_{j_{i_1}} \ldots v_{j_{i_{p}}}
%\Big)
A^{n-p}_K
%=
%\sum 
%\limits_{p=1}^{n-1}
%\left(
%(-1)^{p+1}
%\sigma_p
%\sum
%\limits_{1 \ioe i_1 < \ldots < i_{p} \ioe n-1}
%v_{j_{i_1}} \ldots v_{j_{i_{p}}}
%\right)
%(-1)^{n-p-1}
%\sum 
%\limits_{i=1}^{n-1} 
%\alpha_{j_i}
%v_{j_i}^{n-p-1}
%c_{j_i}
%\\
=
(-1)^{n-1}
\sum 
\limits_{i=1}^{n-1} 
\left( 
\lambda_{j_i} 
- 
\frac{
\lambda_1+ \lambda_K
}
{
2}
\right)
c_{j_i}
\left(
\sum 
\limits_{p=1}^{n-1}
(-1)^p
\sigma_p
v_{j_i}^{n-p-1}
%\sum
%\limits_{1 \ioe i_1 < \ldots < i_{p} \ioe n-1}
%v_{j_{i_1}} \ldots v_{j_{i_{p}}}
\right).
\end{equation*}
Yet, Vieta's formulas give that, for all $x \in \R$, 
\begin{equation*}
(x- v_{j_1}) \ldots (x-v_{j_{n-1}})
=
x^{n-1} 
+
\sum 
\limits_{p=1}^{n-1}
(-1)^p 
\sigma_p
%\sum
%\limits_{1 \ioe i_1 < \ldots < i_{p} \ioe n-1}
%v_{j_{i_1}} \ldots v_{j_{i_{p}}}
x^{n-1-p}
.
\end{equation*}
%So, for all $i=1, \ldots, n-1$, 
%\begin{equation*}
%v_{j_i}^{n-1}
%=
%- 
%\sum 
%\limits_{p=1}^{n-1}
%\left(
%(-1)^p
%\sum
%\limits_{1 \ioe i_1 < \ldots < i_{p} \ioe n-1}
%v_{j_{i_1}} \ldots v_{j_{i_{p}}}
%\right)
%v_{j_i}^{n-1-p}
%\end{equation*}
Using this formula for $x=v_{j_i}$ for every $i \in \{1, \ldots, n-1 \}$, one gets,
\begin{equation*}
\sum 
\limits_{p=1}^{n-1}
%\Big(
%(-1)^{p+1}
%\sum
%\limits_{1 \ioe i_1 < \ldots < i_{p} \ioe n-1}
%v_{j_{i_1}} \ldots v_{j_{i_{p}}}
%\Big)
\sigma_p
A^{n-p}_K
=
(-1)^{n-1} 
\sum \limits_{i=1}^{n-1}
\left( 
\lambda_{j_i} 
- 
\frac{
\lambda_1+ \lambda_K
}
{
2}
\right)
c_{j_i}
(
-v_{j_i}^{n-1}
)
=
-A^n_K,
\end{equation*}
using once again \eqref{coeff_Apk_fin}, which concludes the proof of \eqref{eq} and thus the proof. 
\end{proof}

\begin{prop}
\label{termesdebord}
Under (H1)$_K$-(H2)$_{K, n}$, the solution $\psi$ of \eqref{Schrodinger} with initial data $\varphi_1$
satisfies,
\begin{equation*}
\sum \limits_{p=1}^n 
|u_p(T) | 
= 
\O 
\left( 
\sqrt{T} \|u_n\|_{L^2(0,T)} 
+ 
\|u_1\|^2_{L^2(0,T)} 
+ 
\| (\psi-\psi_1)(T) \|_{L^2(0,1)}
\right).
\end{equation*}
\end{prop}
\begin{proof}
By \cref{existence_J}, there exists $J \subset \N^*$ of cardinal $n$ such that for all $j \in J$, $c_j \neq 0$. Besides, thanks to (H1)$_K$, $J \subset \N^* - \{1, K \}$.  Using the explicit form of $\Psi$ \eqref{order1explicit}, one gets, for $j \in J$,
\begin{multline*}
\langle (\psi-\psi_1)(T), \varphi_j e^{-i \lambda_1 T}\rangle
= 
i \langle \mu \varphi_1, \varphi_j \rangle \int_0^T u(t) e^{ i (\lambda_j - \lambda_1)(t-T)} dt
\\
+
\O
\Big(
\sup_{j \in J}
\left|
\langle
(\psi- \psi_1 - \Psi)(T),
\varphi_j
\rangle
\right|
\Big).
\end{multline*}
Yet, as $j \neq 1$, doing $n$ integrations by parts and using Cauchy-Schwarz inequality, one gets 
\begin{equation*}
 \int_0^T u(t) e^{i(\lambda_j-\lambda_1)(t-T)} dt 
 =
 \sum \limits_{p=1}^n
 \left(
 -i [\lambda_j-\lambda_1] 
 \right)^{p-1} 
 u_p(T) 
 + 
 \O(\sqrt{T} \|u_n\|_{L^2}).
 %= u_1(T) - i[\lambda_j- \lambda_1]u_2(T) + \ldots + (-1)^{n-1} [i(\lambda_j-\lambda_1)]^{n-1} u_n(T)
 %\\&+(-1)^n [i(\lambda_j-\lambda_1)]^n \int_0^T u_n(t) e^{i(\lambda_j- \lambda_1)(t-T)} dt, \\
 %\\= u_1(T) - i[\lambda_j- \lambda_1]u_2(T) + \ldots + (-1)^{n-1} [i(\lambda_j-\lambda_1)]^{n-1} u_n(T) + O(\sqrt{T} \|u_n\|_{L^2}).
 \end{equation*}
Using \eqref{estim_rest_quad} to estimate the linear remainder $\psi-\psi_1- \Psi$ (as $J$ is finite), writing 
$U:=(u_p(T))_{p=1 \ldots n}$
and
$
V:=
( 
\left(
 -i [\lambda_j-\lambda_1] 
 \right)^{p-1} 
 )_{(j,p) \in J \times \{1, \ldots, n\}}
 $, the first equation can be written as
\begin{equation*}
V U 
= 
\O
\left( 
\sqrt{T} \|u_n\|_{L^2} 
+ 
\|u_1\|^2_{L^2} 
+
\| (\psi-\psi_1)(T) \|_{L^2}
\right).
\end{equation*}
The invertibility of the Vandermonde matrix V concludes the proof. 
\end{proof}
Therefore, in \eqref{expr_xi_2}, neglecting the boundary terms $(u_p(T))_{p=1, \ldots, n}$ thanks to \cref{termesdebord} and neglecting the other boundary terms thanks to Cauchy-Schwarz inequality and the boundness of $h$, one gets,
\begin{equation}
\label{conclu:quad}
\Im 
\langle \xi(T), \varphi_K e^{-i \lambda_1 T} \rangle
= Q_n(u_n) 
+
\O \left(
T \| u_n\|^2_{L^2}
+
\| u _1\|^4_{L^2}
+
\| (\psi-\psi_1)(T) \|_{L^2}^2
\right),  
\end{equation}
where $Q_n$ is the quadratic form \eqref{Qn}, which is coercive as stated in \cref{coercivityn}. 

\section{Proof of the obstructions caused by quadratic integer drifts}
\label{proof_obstructions}
From now on, we assume that (H1)$_K$-(H2)$_{K, n}$ hold and we work in the asymptotic $u$ small in $W^{-1, \infty}$ if $n=1$ and small in $H^{2n-3}$ if $n \soe 2$. The proof of \cref{obstruction_n} consists in describing what happens at each order of the expansion of the solution. Under (H1)$_K$, recalling the explicit computation of $\Psi$ given in \eqref{order1explicit}, we have
$$\langle \Psi(T), \varphi_K e^{-i \lambda_1 T}\rangle =0.$$
Thus, the study of the quadratic term given by \eqref{conclu:quad} and estimate \eqref{estim_rest_cub} of the cubic remainder entail
%$$\langle \Psi(T), \varphi_K e^{-i \lambda_1 T}\rangle =0.$$
%
%\smallskip \noindent \textit{At first order.} Under the assumption (H2), recalling the explicit computation of $\Psi$ given in \eqref{order1explicit},  
%$$\langle \Psi(T), \varphi_K e^{-i \lambda_1 T}\rangle =0.$$
%
%\smallskip \noindent \textit{The quadratic term.} In \red{section}, we saw that, under the assumption (H3)$_n$ and for specific motions of $\psi$ given in \cref{termesdebord},  the quadratic term has a drift quantified by the $H^{-n}$-norm of the control in the way of \eqref{conclu:quad}.
%\begin{equation*}
%\Im \langle \xi(T), \varphi_K e^{i \lambda_1 T}\rangle 
%= Q_n(u_n)+ O\left(T \|u_n\|^2_{L^2} + \|u_1\|^3_{L^2} \right).
%\end{equation*}
%
%\smallskip \noindent \textit{The cubic remainder.} With \cref{rcS} (as we have the following Sobolev embedding : for $n\geqslant 2$, $H^{2n-3}(0,T) \hookrightarrow W^{-1, \infty}$), that the cubic remainder is estimated by $\|u_1\|^3_{L^2}$.
%%
%
%\smallskip \noindent \textit{Conclusion.} At the end, 
\begin{align}
\notag
\Im \langle \psi(T), \varphi_K e^{-i \lambda_1 T}\rangle 
&=
\Im \langle \xi(T), \varphi_K e^{-i \lambda_1 T}\rangle
+
\O 
\left(
\langle
(\psi
- \psi_1
-\Psi
-\xi)
(T)
, 
\varphi_K
\rangle
\right)
%\\
%&=
%Q_n(u_n) 
%+ \O
%\left(
%T \|u_n\|^2_{L^2} 
%+
%\|u_1\|^4_{L^2} 
%+ 
%\|u_1\|^3_{L^2} 
%+
%\| (\psi-\psi_1)(T) \|_{L^2}^2
%\right)
\\
\label{conclu}
&= 
Q_n(u_n) 
+\O
\left(
T \|u_n\|^2_{L^2} 
+
\|u_1\|^3_{L^2}
+
\| (\psi-\psi_1)(T) \|_{L^2}^2
\right),
\end{align}
thanks to \cref{def_O_bis} of $\O$. 

\medskip \noindent \emph{First obstruction.} When $n=1$, one directly gets that
\begin{equation*}
\left| 
\Im  
\langle \psi(T), \varphi_K e^{-i \lambda_1 T}\rangle 
- Q_1(u_1) 
\right| 
=
\O
\left(
\left(
T+ \|u_1\|_{L^{\infty}}
\right) 
\|u_1\|^2_{L^2} 
+
\| (\psi-\psi_1)(T) \|_{L^2}^2
\right). 
\end{equation*}
As in \cref{proof_obstru_dim_finie}, expanding the notation of $\O$, this leads to the existence of $C>0$ such that for all $A \in (0, \frac{|A^1_K|}{4})$, there exists $T^*>0$ such that for all $T \in (0, T^*)$, there exists $\eta>0$ such that for all $u \in L^2(0,T)$ with $\|u_1\|_{L^{\infty}(0,T)} < \eta$, 
\begin{equation*}
\left| 
\Im
\langle 
\psi(T), 
\varphi_K e^{-i \lambda_1 T}
\rangle  
- 
Q_{1}(u_1) 
\right|  
\ioe 
\left(
\frac{|A^1_K|}{4}
-A
\right)
\|u_1\|^2_{L^2}
+
C
\| (\psi-\psi_1)(T)\|_{L^2}^2
.
\end{equation*}
This inequality, together with the coercivity of $Q_1$ given in \cref{coercivityn}, concludes the proof of \cref{obstruction_n} for $n=1$. 

\medskip \noindent \emph{Second obstruction and the following.}  When $n \soe 2$, it remains to prove that the cubic term $\|u_1\|^3_{L^2}$ can be, in a good functional setting, neglected in front of the drift $\|u_n\|^2_{L^2}$. This is done thanks to Gagliardo-Nirenberg inequalities \cite[Theorem p.125]{N59}.
\begin{lem}
Let $n \in \mathbb{N}, n\soe 2$. There exists $C> 0$ such that for all $T> 0$ and $v \in H^{3(n-1)}(0,T)$, we have, 
$$ 
\| v^{(n-1)}\|^3_{L^2(0,T)} 
\ioe
C \| v ^{3(n-1)} \|_{L^2(0,T)} \|v \|^2_{L^2(0,T)} + C T^{-3(n-1)} \|v\|^3_{L^2(0,T)}.$$
\end{lem}
\noindent Applying this inequality to $u_n$, we have 
\begin{equation*}
\|u_1\|^3_{L^2(0,T)} 
%&\ioe 
%C  \| u_n^{3(n-1)} \|_{L^2(0,T)} \|u_n \|^2_{L^2(0,T)} + C T^{-3(n-1)} \|u_n\|^3_{L^2(0,T)}, 
%\\
\ioe 
C \left( \|u^{(2n-3)}\|_{L^2(0,T)} + T^{-2n +3} \|u\|_{L^2(0,T)} \right)\|u_n\|^2_{L^2(0,T)}.
\end{equation*}
So, this Gagliardo-Nirenberg inequality, together with \eqref{conclu} gives,
\begin{multline*}
\left| 
\Im  
\langle \psi(T), \varphi_K e^{-i \lambda_1 T}\rangle 
- Q_n(u_n) 
\right| 
\\
=
\O
\left(
\left[T+ \|u^{(2n-3)}\|_{L^2} + T^{-2n +3} \|u\|_{L^2} \right] 
\|u_n\|^2_{L^2} 
+
\| (\psi-\psi_1)(T) \|_{L^2}^2
\right).
\end{multline*}
%
%As in \cref{proof_obstru_dim_finie}, expanding the notation of $\O$, this lead to, for all $\eps>0$, the existence of $T^*>0$ such that for all $T \in (0, T^*)$, there exists $\eta>0$ such that for all $u \in H^{2n-3}(0,T)$ with $\|u\|_{H^{2n-3}(0,T)} \ioe \eta$, 
%
%Expanding the definition of the notation $\O$, one gets that for all $\varepsilon > 0$, there exists $C_1, T_1> 0$ such %that for all $T \in (0,T_1)$, there exists $\eta_1 > 0$ such that for all $u \in H^{2n-3}(0,T)$ with $\|u_1\|_{L^{\infty}} \ioe %\eta_1$, 
%\begin{equation*}
%\left| 
%\Im 
%\langle \psi(T), \varphi_K e^{-i \lambda_1 T}\rangle
%- Q_{n}(u_n) \right|  
%\ioe 
%C_1 
%\left(
%T+ \|u^{(2n-3)}\|_{L^2} + T^{-\frac{5}{2}n +3} \|u\|_{L^2} 
%\right) 
%\|u_n\|^2_{L^2}
%.
%\end{equation*}
%
%Let $\varepsilon > 0$. Let $T^* := \min( T_1, \frac{\epsilon}{3 C_1})$, $T \in (0, T^*)$ and  $\eta := \min( \eta_1, %\frac{\epsilon}{3 C_1}, \frac{\epsilon}{3 C_1} T^{\frac{5}{2}n-3}).$ Then, for all $u \in H^{2n-3}(0,T)$ such that $\|u\|%_{H^{2n-3}} \ioe \eta$, 
%\begin{equation*}
%\left| 
%\Im
%\langle 
%\psi(T), 
%\varphi_K e^{-i \lambda_1 T}
%\rangle  
%- 
%Q_{n}(u_n) 
%\right|  
%\ioe 
%\varepsilon \|u_n\|^2_{L^2(0,T)}
%+
%C
%\| (\psi-\psi_1)(T)\|_{L^2(0,1)}^2
%.
%\end{equation*}
As before, expanding the notation of $\O$, it means that there exists $C, T_1 >0$ such that for all $T \in (0, T_1)$, there exists $\eta_1>0$ such that for all $u \in H^{2n-3}(0,T)$ with $\| u_1\|_{L^{\infty}(0,T)} < \eta_1$, 
\begin{multline*}
\left| 
\Im  
\langle \psi(T), \varphi_K e^{-i \lambda_1 T}\rangle 
- Q_n(u_n) 
\right| 
\\
\ioe 
C
\left(
\left[T+ \|u^{(2n-3)}\|_{L^2} + T^{-2n +3} \|u\|_{L^2} \right] 
\|u_n\|^2_{L^2} 
+
\| (\psi-\psi_1)(T) \|_{L^2}^2
\right).
\end{multline*}
Let $A \in (0, \frac{|A^n_K|}{4})$, $T^*:= \min(T_1, \frac{|A^n_K|}{12C}-\frac{A}{3C})$, $T \in (0,T^*)$ and $\eta:=\min(\eta_1, \frac{|A^n_K|}{12C}-\frac{A}{3C}, \left(\frac{|A^n_K|}{12C}-\frac{A}{3C}\right)T^{2n -3})$. Then, for all  $u \in H^{2n-3}(0,T)$ with $\| u \|_{H^{2n-3}(0,T)} < \eta$, one has
\begin{equation*}
\left| 
\Im
\langle 
\psi(T), 
\varphi_K e^{-i \lambda_1 T}
\rangle  
- 
Q_{n}(u_n) 
\right|  
\ioe 
\left(
\frac{|A^n_K|}{4}
-A
\right)
\|u_n\|^2_{L^2}
+
C
\| (\psi-\psi_1)(T)\|_{L^2}^2
.
\end{equation*}
This inequality, together with the coercivity of $Q_n$ given in \cref{coercivityn} concludes the proof of \cref{obstruction_n} for $n\soe 2$.

\begin{rem}
Notice that for $n \soe 2$, the smallness assumption on the control depends on the final time $T$, which is not the case for $n=1$. Such a phenomenon already appears in finite dimension (see \cite[Section 2.4.4]{BM18} for example). 
\end{rem}

\appendix
\section{Existence of $\mu$ satisfying (H1)$_K$-(H2)$_{K, n}$}
\label{existence_mu}
The goal of this appendix is to prove the following theorem.
\begin{thm}
For all $K \in \N^*$, $K \soe 2$ and $n \in \N^*$, there exists $\mu \in C^{\infty}_c(0,1)$ satisfying (H1)$_K$-(H2)$_{K, n}$. 
\end{thm}
Notice that in all this section, the coefficients $(A^p_K)_{p \in \N^*}$ given in \eqref{coeff_quad_nul} and \eqref{coeff_quad_non_nul} are seen as quadratic forms with respect to $\mu$.

\subsection{Strategy of the proof}
The existence of $\mu \in C^{\infty}_c(0,1)$ satisfying (H1)$_K$-(H2)$_{K, n}$ can be brought down to three steps. 
\begin{itemize}
\item First, instead of working with the series given in \eqref{coeff_quad_nul} and \eqref{coeff_quad_non_nul}, we prove that the coefficients $(A^p_K)_{p \in \N^*}$ can be namely written as a sum of $L^2(0,1)$-scalar products of two derivatives of $\mu$ with $\varphi_1 \varphi_K$. 
\item Then, putting together the terms of the same order, we prove that the quadratic forms $A^p_K$ can be written as 
\begin{equation*}
A^p_K( \mu) 
\approx
\langle 
{\mu^{(2p-1)}}^2 \varphi_1, \varphi_K
\rangle.
\end{equation*}  
\item Therefore, one can construct oscillating functions $\mu$ such that
\begin{equation*}
\left( 
\langle 
\mu 
\varphi_1, 
\varphi_K 
\rangle, 
A^1_K(\mu),
\ldots, 
A^n_K(\mu)
\right)
\approx 
\left(
0,
0,
\ldots, 
0,
\pm 1
\right).
\end{equation*}
The negligible terms are dealt with a Brouwer fixed-point theorem. 
\end{itemize}

\subsection{Computations on Lie brackets}
In this appendix, we work with $\mu$ in $C^{\infty}_c(0,1)$ to ensure the well-posedness of all the Lie brackets considered and also to not worry about boundary terms when performing integrations by parts. However, the following results hold with weaker regularity and fewer boundary conditions on $\mu$. 
\begin{prop}
\label{expr_via_LB}
For all $p \in \N^*$, the quadratic form $A^p_K$ defined in \eqref{coeff_quad_nul} can be written as 
\begin{equation}
\label{expr_ApK_crochet}
\forall \mu \in C_c^{\infty}(0,1), 
\quad
A^p_K(\mu)
= 
\dfrac{(-1)^{p-1}}{2} 
\langle  
[\ad_A^{p-1}(\mu), \ad_A^{p}(\mu)] \varphi_1,  
\varphi_K\rangle.
\end{equation}
\end{prop}

\begin{proof}
First, one can prove by induction on $q$ in $\N$ that for all $a, b \in \N^*$, 
\begin{equation}
\label{formule_ad_q}
(\lambda_a- \lambda_b)^q 
\langle \mu \varphi_a, \varphi_b \rangle 
= 
(-1)^q 
\langle \ad^q_{A}(\mu) \varphi_a, \varphi_b \rangle. 
\end{equation}
Indeed, it holds for $q=0$ by definition of $\ad^0_A(\mu)$. Moreover, as the functions $\varphi_j$ are the eigenvectors of $A$ which is a symmetric operator, the heredity follows after noticing that 
\begin{equation*}
\langle \ad^{q+1}_{A}(\mu) \varphi_a, \varphi_b \rangle
%=
%\langle [A, \ad^{q}_{A}(\mu)] \varphi_a, \varphi_b \rangle
=
\langle (A - \lambda_a \Id) \ad^{q}_{A}(\mu) \varphi_a, \varphi_b \rangle
%=
%\langle \ad^{q}_{A}(\mu) \varphi_a, (A - \lambda_a \Id) \varphi_b \rangle
= 
-(\lambda_a-\lambda_b) \langle \ad^{q}_{A}(\mu) \varphi_a, \varphi_b \rangle.
\end{equation*}
Therefore, from \eqref{coeff_quad_nul} and \eqref{formule_ad_q}, one deduces that $2 (-1)^{p-1} A^p_K$ is given by
\begin{align*}
%2 (-1)^p A^p_K 
&
\sum \limits_{j=1}^{+\infty}  
(\lambda_j - \lambda_1)^p 
(\lambda_K - \lambda_j)^{p-1}
c_j
-
\sum \limits_{j=1}^{+\infty}  
(\lambda_j - \lambda_1)^{p-1} 
(\lambda_K - \lambda_j)^{p}
c_j
\\
&=
(-1)^{p-1}
\left(
\sum \limits_{j=1}^{+\infty}  
\langle \ad^p_A(\mu) \varphi_1, \varphi_j \rangle
\langle \ad^{p-1}_A(\mu) \varphi_K, \varphi_j \rangle
+
\sum \limits_{j=1}^{+\infty}  
\langle \ad^{p-1}_A(\mu) \varphi_1, \varphi_j \rangle
\langle \ad^p_A(\mu) \varphi_K,  \varphi_j \rangle
\right)
\\
&=
(-1)^{p-1}
\left(
\langle \ad^p_A(\mu) \varphi_1, \ad^{p-1}_A(\mu) \varphi_K \rangle
+
\langle \ad^{p-1}_A(\mu) \varphi_1, \ad^{p}_A(\mu) \varphi_K \rangle
\right),
\end{align*}
which gives \eqref{expr_ApK_crochet} using the symmetry/skew-symmetry of the operators $\ad^{k}_A(\mu)$. 
\end{proof}

\begin{prop}
\label{semi-explicit_formula}
For all $p \in \N^*$, there exists a constant $C>0$ and $Q_p$ a quadratic form such that for all $\mu \in C_c^{\infty}(0,1)$, 
\begin{equation}
\label{fq_reduction}
A^p_K(\mu)
=
\langle 
{\mu^{(2p-1)}}^2 
\varphi_1,
\varphi_K
\rangle
+
Q_p(\mu)
\quad
\text{ with }
\quad 
| Q_p(\mu)| \ioe C \| \mu\|^2_{H^{2p-2}(0,1)}.
\end{equation}
\end{prop}

\begin{proof} 
Let $\mu \in C_c^{\infty}(0,1)$.

\medskip \noindent \emph{Step 1: Computations of the Lie brackets $(\ad^p_A(\mu))_{p \in \N}$.} First, one can prove by induction on $p \in \N$ that
\begin{equation}
\label{expr_ad_p}
\forall f \in C^{\infty}([0,1]), 
\quad
\ad_{A}^p(\mu) f = \sum \limits_{k=0}^p \alpha_k^p \mu^{(2p-k)} f^{(k)},
\end{equation}
where the coefficients $(\alpha_k^p)_{k=0, \ldots, p}$ are defined by induction as $\alpha_0^0:=1$ and for all $p \in \N^*$,  
\begin{equation}
\label{def_alpha}
\alpha_0^{p+1}:=-\alpha_0^p,
\quad 
\alpha_{p+1}^{p+1}:=-2 \alpha_p^p
\quad 
\text{ and }
\quad
\forall k=1, \ldots, p, 
\quad
\alpha_k^{p+1}:=-\alpha^p_k -2 \alpha_{k-1}^p.
\end{equation}
Indeed, \eqref{expr_ad_p} holds for $p=0$ by definition of $\ad^0_A(\mu)$. Moreover, if \eqref{expr_ad_p} is true for some $p \in \N^*$, then
\begin{align*}
\ad^{p+1}_A(\mu)f 
=
[A, \ad^p_A(\mu)]f
&=
-
\sum \limits_{k=0}^p 
\alpha_k^p 
\left(
\mu^{(2p-k)} f^{(k)}
\right)^{(2)}
+
\sum \limits_{k=0}^p 
\alpha_k^p 
\mu^{(2p-k)} 
f^{(k+2)}
\\
&=
-
\sum \limits_{k=0}^p 
\alpha_k^p 
\mu^{(2p-k+2)} 
f^{(k)}
-
2
\sum \limits_{k=0}^p 
\alpha_k^p 
\mu^{(2p-k+1)} 
f^{(k+1)},
\end{align*}
using Leibniz formula. And the heredity holds after a shift of indexes in the second sum by definition \eqref{def_alpha} of the $(\alpha_k^p)_{k=0, \ldots, p}$. Moreover, one can also prove by induction that for all $p \in \N$, 
\begin{equation}
\label{somme_coeff}
\sum \limits_{k=0}^p (-1)^k \alpha_k^p =1. 
\end{equation}
Indeed, it is true for $p=0$ by definition of $\alpha_0^0$. Moreover, if \eqref{somme_coeff} is true for some $p \in \N^*$, then
\begin{align*}
\sum \limits_{k=0}^{p+1}
(-1)^k 
\alpha_{k}^{p+1}
&=
\alpha_0^{p+1}
-
\sum \limits_{k=1}^{p}
(-1)^k
(\alpha^p_k + 2 \alpha_{k-1}^p)
+
(-1)^{p+1} \alpha_{p+1}^{p+1}
%\\
%&=
%\alpha_0^{p+1}
%-
%\sum \limits_{k=1}^{p}
%(-1)^k
%\alpha^p_k 
%-2
%\sum \limits_{k=0}^{p-1}
%(-1)^{k+1}
%\alpha^p_k
%+
%(-1)^{p+1} \alpha_{p+1}^{p+1}
\\
&
=
\alpha_0^{p+1}
-
(1 - \alpha_0^p)
+2
(1- (-1)^p \alpha^p_p)
+
(-1)^{p+1} \alpha_{p+1}^{p+1}
%\\
%&=
%\alpha_0^{p+1} 
%+ 
%\alpha_0^p 
%+ 
%(-1)^{p+1} 
%\left( 
%2 \alpha_p 
%+ 
%\alpha_{p+1}^{p+1} 
%\right)
=1,
\end{align*}
using the definition \eqref{def_alpha} of the coefficients $(\alpha_k^p)_{k=0, \ldots, p}$ and the statement \eqref{somme_coeff} for $p$. 

\medskip \noindent \emph{Step 2: Semi-explicit formula for the Lie brackets $([\ad^{p-1}_A(\mu), \ad^{p}_A(\mu)])_{p \in \N^*}$}. Using the explicit formula \eqref{expr_ad_p} and Leibniz formula, one can compute that
\begin{multline}
\label{expr:W_p}
\left[ 
\ad_{A}^{p-1}(\mu) 
, \ad_{A}^p(\mu) 
\right] f
%\sum \limits_{k=0}^{p-1} 
%\alpha_k^{p-1}
%\mu^{(2p-2-k)}
%\left(
%\ad_{A}^{p}(\mu)f
%\right)^{(k)}
%-
%\sum \limits_{k=0}^{p} 
%\alpha_k^{p}
%\mu^{(2p-k)}
%\left(
%\ad_{A}^{p-1}(\mu)f
%\right)^{(k)}
%\\
=
\sum \limits_{k=0}^{p-1} 
\sum \limits_{n=0}^p
\sum \limits_{m=0}^k
\binom{k}{m}
\alpha_k^{p-1}
\alpha_n^p
\mu^{(2p-2-k)}
\mu^{(2p-n+k-m)}
f^{(n+m)}
\\
-
\sum \limits_{k=0}^{p} 
\sum \limits_{n=0}^{p-1}
\sum \limits_{m=0}^k
\binom{k}{m}
\alpha_k^{p}
\alpha_n^{p-1}
\mu^{(2p-k)}
\mu^{(2p-2-n+k-m)}
f^{(n+m)}. 
\end{multline}
Many terms in these sums can be neglected. Indeed, for all $a, b \in \N$, for all $f \in C^{\infty}([0,1])$, there exists $C>0$ such that 
\begin{equation}
\label{estim_ps}
\left|
\langle 
\mu^{(a)}
\mu^{(b)}
,
f
\rangle
\right|
\ioe 
C
\| \mu \|^2_{
H^{
\lfloor 
\frac{a+b}{2}
\rfloor 
}
}
.
\end{equation}
\noindent 
As for all $(n, m) \neq (0,0)$, 
$
\lfloor
\frac{4p-2-n-m}{2} 
\rfloor
\ioe 
2p-2,
$
thanks to \eqref{estim_ps}, every term of the two sums linked to some $(k,n,m)$ such that $(n,m) \neq (0,0)$ can be bounded by $\| \mu\|_{H^{2p-2}}^2$. Therefore, one gets the existence of $C>0$ such that 
\begin{multline}
\label{expr_term_dom}
\Big|
\left\langle
\left[ 
\ad_{A}^{p-1}(\mu) 
, \ad_{A}^p(\mu) 
\right]
\varphi_1,
\varphi_K
\right\rangle
-
\sum \limits_{k=0}^{p-1} 
\alpha_k^{p-1}
\alpha_0^p
\langle
\mu^{(2p-2-k)}
\mu^{(2p+k)}
\varphi_1, 
\varphi_K 
\rangle
\\+
\sum \limits_{k=0}^{p} 
\alpha_k^{p}
\alpha_0^{p-1}
\langle
\mu^{(2p-k)}
\mu^{(2p-2+k)}
\varphi_1,
\varphi_K
\rangle 
\Big|
\ioe 
C 
\| 
\mu
\|_{H^{2p-2}}^2
.
\end{multline}
Moreover, every remaining scalar product in \eqref{expr_term_dom} can be written as 
$
\pm 
\langle
{\mu^{(2p-1)}}^2
\varphi_1,
\varphi_K
\rangle
$
up to some negligible terms, as stated by the following estimates,
\begin{equation*}
\forall k=-p, \ldots,  p-1,
\quad
\left|
\langle
\mu^{(2p-2-k)}
\mu^{(2p+k)}
\varphi_1, 
\varphi_K 
\rangle
-
(-1)^{k+1}
\langle
{\mu^{(2p-1)}}^2
\varphi_1, 
\varphi_K 
\rangle
\right|
\ioe 
C
\| 
\mu
\|_{H^{2p-2}}^2
.
%\\
%\forall k=0, \ldots,  p,
%\quad
%\left|
%\langle
%\mu^{(2p-k)}
%\mu^{(2p-2+k)}
%\varphi_1, 
%\varphi_K 
%\rangle
%-
%(-1)^{k+1}
%\langle
%{\mu^{(2p-1)}}^2
%\varphi_1, 
%\varphi_K 
%\rangle
%\right|
%&\ioe 
%C
%\| 
%\mu
%\|_{H^{2p-2}}^2.
\end{equation*}
These estimates are proved by integrations by parts and using \eqref{estim_ps}. 
Therefore, one gets the existence of $C>0$ such that 
\begin{multline}
\Big|
\left\langle
\left[ 
\ad_{A}^{p-1}(\mu) 
, \ad_{A}^p(\mu) 
\right]
\varphi_1,
\varphi_K
\right\rangle
\\-
\left(
\alpha_0^p 
\sum \limits_{k=0}^{p-1}
(-1)^{k+1}
\alpha_k^{p-1}
-
\alpha_0^{p-1}
\sum \limits_{k=0}^{p}
(-1)^{k+1}
\alpha_k^p
\right)
\langle
{\mu^{(2p-1)}}^2
\varphi_1,
\varphi_K
\rangle 
\Big|
\ioe 
C 
\| 
\mu
\|_{H^{2p-2}}^2
.
\end{multline}
Using \eqref{somme_coeff} and that $\alpha_0^p=(-1)^p$ by the first equality of \eqref{def_alpha}, one gets that
\begin{equation*}
\alpha_0^p 
\sum \limits_{k=0}^{p-1}
(-1)^{k+1}
\alpha_k^{p-1}
-
\alpha_0^{p-1}
\sum \limits_{k=0}^{p}
(-1)^{k+1}
\alpha_k^p
=
2 (-1)^{p-1},
\end{equation*}
which concludes the proof of \eqref{fq_reduction} using \eqref{expr_ApK_crochet}. 
%\begin{multline*}
%%\left|
%\sum \limits_{k=0}^{p-1} 
%\alpha_k^{p-1}
%\alpha_0^p
%\langle
%\mu^{(2p-2-k)}
%\mu^{(2p+k)}
%\varphi_1, 
%\varphi_K 
%\rangle
%-
%\sum \limits_{k=0}^{p} 
%\alpha_k^{p}
%\alpha_0^{p-1}
%\langle
%\mu^{(2p-k)}
%\mu^{(2p-2+k)}
%\varphi_1,
%\varphi_K
%\rangle
%- 2(-1)^{p-1} 
%\right|
%\\
%\ioe
%C \| \mu \|^2_{H^{2p-2}(0,1)}.
%\end{multline*} 
%Therefore, together with \eqref{expr_term_dom}, this lead to \eqref{fq_reduction}.
\end{proof}

\subsection{Proof of the existence of $\mu$ satisfying (H1)$_K$-(H2)$_{K, n}$}

\begin{thm}
Let $K \in \N^*$, $K \soe 2$ and $\overline{x} \in (0,1)$ such that $\sin(K \pi \overline{x})=0$. There exists $\delta>0$ such that for every $n \in \N^*$, for every $J^+$ (resp. $J^-$) open interval of $( \overline{x}, \overline{x}+ \delta)$ (resp. $(\overline{x}- \delta, \overline{x})$), there exists $\mu_n^+$ (resp. $\mu_n^-$) in $C^{\infty}_c(0,1)$ supported on $J^+$ (resp. $J^-$) such that 
\begin{equation*}
\langle \mu_n^{\pm} \varphi_1, \varphi_K \rangle
=
A_K^1( \mu_n^{\pm})
= 
\ldots 
=
A_K^{n-1}( \mu_n^{\pm})
=
0
\quad 
\text{ and }
\quad 
A_K^n( \mu_n^{\pm}) = \pm 1.
\end{equation*}
\end{thm}
\begin{proof}
As $\varphi_1 >0$ on $(0,1)$, by definition of $\overline{x}$ and by continuity, there exists $\delta>0$ such that $\varphi_1 \varphi_K>0$ on $( \overline{x}, \overline{x}+ \delta)$ and  $\varphi_1 \varphi_K<0$ on $(\overline{x}- \delta, \overline{x})$ (or conversely, but it works the same). Let us prove the statement by induction on $n$. 

\medskip \noindent \emph{Initialization.} Let $J^+$ be an open interval of $( \overline{x}, \overline{x}+ \delta)$. One can construct $\mu_1^+ \neq 0$ supported on $J^+$ such that $\langle \mu_1^+ \varphi_1, \varphi_K\rangle =0$. Moreover, looking at \eqref{expr_ApK_crochet},  $A^1_K( \mu_1^+)= \langle (\mu_1^+)'^2 \varphi_1, \varphi_K \rangle >0$ and thus after rescaling satisfies $A^1_K(\mu_1^+) = 1$. One can construct $\mu_1^-$ similarly.  

\medskip \noindent \emph{Heredity.} Assume that the statement of the theorem holds for $1, \ldots, n$ and we prove it for $n+1$. Let $J^+$ (resp. $J^-$) open interval of $( \overline{x}, \overline{x}+ \delta)$ (resp. $(\overline{x}- \delta, \overline{x})$). There exists $(I^{\pm}_i)_{i=0, \ldots, n}$ open intervals of $J^{\pm}$, two by two disjoints. By induction, for all $i=1, \ldots, n$, there exists $\mu_i^{\pm}$ supported on $I_i^{\pm}$ such that 
\begin{equation}
\label{constru_mui}
\langle \mu_i^{\pm} \varphi_1, \varphi_K \rangle
=
A_K^1( \mu_i^{\pm})
= 
\ldots 
=
A_K^{i-1}( \mu_i^{\pm})
=0
\quad 
\text{ and }
\quad 
A_K^i( \mu_i^{\pm}) = \pm 1.
\end{equation}
If needed, we denote by $\mu_i^0 \equiv 0$ for all $i=1, \ldots, n$. Moreover, one can also choose $\mu_0$ supported on $I_0^{\pm}$ such that 
\begin{equation}
\label{constru_mu0}
\langle \mu_0 \varphi_1, \varphi_K \rangle=1. 
\end{equation}
We prove that $F : \mu \mapsto (\langle \mu \varphi_1, \varphi_K \rangle, A^1_K(\mu), \ldots, A^{n+1}_K(\mu))$ is onto and thereby prove the heredity.  

\medskip \noindent \emph{Step 1: Surjectivity of the first $n+1$-components of $F$.} Let $a=(a_0, a_1, \ldots, a_{n}) \in \R^{n+1}$. As $\mu_0$ and $(\mu_i^{\pm})_{i=1, \ldots, n}$ have all disjoint supports and satisfy respectively \eqref{constru_mu0} and \eqref{constru_mui}, the function
\begin{equation*}
\hat{\mu}_a:= 
a_0 \mu_0
+
\sum \limits_{i=1}^n
\sqrt{
\left|
\alpha_i(a)
\right|
} 
\mu_i^{\sign(\alpha_i(a))},
\end{equation*}
where the coefficients $(\alpha_i)_{i=1, \ldots, n}$ are fitted in the following way, 
\begin{equation*}
\alpha_1(a):= a_1
\quad
\text{ and }
\
\forall p \in \{2, \ldots, n\}, 
\
\alpha_p(a) 
:= 
a_p
-
a_0^2
A^p_K(\mu_0)
- 
\sum \limits_{i=1}^{p-1}
| \alpha_i(a) | 
A^p_K
\left( 
\mu_i^{ \sign(\alpha_i(a))}
\right)
,
\end{equation*}
is supported on $\cup_{i=0}^n I_i^{\pm}$ and satisfies
\begin{equation}
\label{onto_n+1}
F(\hat{\mu}_a)=(a_0, a_1, \ldots, a_n, A^{n+1}_K(\hat{\mu}_a)).
\end{equation} 
Besides, the last component of $F(\hat{\mu}_a)$ can be estimated as
\begin{equation}
\label{onto_n+1_bis}
A^{n+1}_K(\hat{\mu}_a)
=
\O
\left(
a_0^2, 
|a_1|, 
\ldots, 
|a_n|
\right).
\end{equation}
Moreover, one can check that $a \mapsto F(\hat{\mu}_a)$ is continuous on $\R^{n+1}$. Indeed, by induction on $p \in \{1, \ldots, n\}$, the functions $a \mapsto \alpha_p(a)$ are continuous and thus the functions $\sign(\alpha_p(\cdot))$ are locally constant. 
%\red{Fauxxxx! Moreover, by induction on $i$, on can prove there exists $\delta>0$ such that for all $i=1, \ldots, n$, $\alpha_i$ is analytic on $B(0, \delta)$ and thus $\alpha_i \neq 0$ on $B(0, \delta) \setminus \{0\}$. And thus, $a \mapsto F(\hat{\mu}_a)$ is continuous on $B(0, \delta) \setminus \{0\}$. Moreover, by \eqref{onto_n+1}, $ F(\hat{\mu}_a) \rightarrow 0$ when $a \rightarrow 0$ so the function can be continuously extended at zero. }

\medskip \noindent  \emph{Step 2: Surjectivity of the last component of $F$.}
%Our goal is to prove that there exists $I_{n+1}^{\pm} \subset J^{\pm} \setminus \cup_{i=1}^n I_i^{\pm}$ and $\eta>0$ such that for all $a \in \R^*$, $|a| < \eta$, there exists $\tild{\mu}_a \in C_c^{\infty}(I_{n+1}^{\pm})$ such that 
%\begin{equation}
%\label{almost_onto_last}
%F(\tild{\mu}_a)=(0, \ldots, 0, a) + \O( |a |^3).
%\end{equation} 
%Moreover, there exists $a^* \neq 0$ and $\rho>0$ such that $a \mapsto F(\tild{\mu}_a)$ is continuous on $B(a^*, \rho)$. 
%
%Version sans continuité
%Our goal is to prove that there exists $I_{n+1}^{\pm} \subset J^{\pm} \setminus \cup_{i=1}^n I_i^{\pm}$ and $\eta>0$ such that for all $a \in \R$, $|a| < \eta$, there exists $\tild{\mu}_a \in C_c^{\infty}(I_{n+1}^{\pm})$ such that $F(\tild{\mu}_a)=(0, \ldots, 0, a) + \O( |a |^3)$.
%
%
Let $a \in \R^*$. Define, 
\begin{equation*}
\tild{\mu}_{a}(x)
:= 
\frac{
|a|^{2n-1}
}
{
\sqrt{
|\varphi_1(x(a)) \varphi_K(x(a))|
}
} 
\
g 
\left(
\frac{x-x(a)}{|a|} 
\right),
\end{equation*}
where $g$ is in $C^{\infty}_c(0,1)$ such that $\int_0^1 g^{(2n-1)}(y)^2 dy=1$ and $x(a)$ is defined by $x(a)=x^+ \1_{a>0}+x^- \1_{a<0}$, where $x^{\pm}$ are in $J^{\pm} \setminus \cup_{i=0}^n I_i^{\pm}$ such that $\varphi_K(x^+) >0$ (and $\varphi_K(x^-) <0)$. First, $\tild{\mu}_a$ is supported on $(x^- -a, x^-)$ if $a<0$ and $(x^+, x^+ +a)$ if $a>0$. Thus, for $a$ small enough, the support of $\tild{\mu}_a$ is in $J^{\pm} \setminus \cup_{i=0}^n I_i^{\pm}$. 
Then, for example when $a>0$, a change of variables and a Taylor expansion with respect to $a$ give,
\begin{align*}
\langle (\tild{\mu}_a^{(2n-1)})^2 \varphi_1, \varphi_K \rangle
&=
\frac{1}{\varphi_1(x^+) \varphi_K(x^+)}
\int_{x^+}^{x^+ + a} 
g^{(2n-1)} 
\left(
\frac{x-x^+}{a} 
\right)^2
\varphi_1(x)
\varphi_K( x)
dx
\\
&=
\frac{a}{\varphi_1(x^+) \varphi_K(x^+)}
\int_0^1 
g^{(2n-1)} 
(y)^2
\varphi_1(x^+ + ay)
\varphi_K(x^+ + ay)
dy
\\
&=
a+\O( |a|)^3.
\end{align*}
Moreover, for all $k=0, \ldots, 2n-2$, similarly, one gets
\begin{equation*}
\| \tild{\mu}_a^{(k)} \|_{L^2(0,1)}^2
=
\frac
{|a|^{4n-2k-1}}
{|\varphi_1(x(a)) \varphi_K(x(a))|}
\int_0^1
g^{(k)} 
(y)^2
dy
=
\O
( |a|^3)
.
\end{equation*}
Hence, by \cref{semi-explicit_formula} giving a semi-explicit formula for the quadratic forms $(A^p_K)_{p \in \N^*}$, one gets 
\begin{equation}
\label{onto_n+2}
F( \tild{\mu}_a) = (0, \ldots, 0, a) + \O( |a|^3). 
\end{equation}
Moreover, the previous computations prove the continuity of $a \mapsto F(\tild{\mu}_a)$ on $\R^*$. Besides, by \eqref{onto_n+2}, $F( \tild{\mu}_a) \rightarrow 0$ when $a \rightarrow 0$ so the map can be extended continuously at 0.

\medskip \noindent  \emph{Step 3: F is onto by Brouwer.} By Step 1 and Step 2, and more precisely with \eqref{onto_n+1}, \eqref{onto_n+1_bis} and \eqref{onto_n+2}, there exists $\delta>0$ such that, for all $a \in \R^{n+2}$ such that $a \neq 0$ and $|a | < \delta$, 
%Let $a=(a_0, \ldots, a_{n+1}) \in \R^{n+2}$ with $|a_{n+1}| \ioe \eta$. By Step 1 and 2, as the functions have disjoints support, taking 
%$\mu_a:=
%\hat{\mu}_a
%+
%\tild{\mu}_{
%a_{n+1} - A^{n+1}_K( \hat{\mu}_a)
%}
%$,
one gets the existence of $C>0$ such that
\begin{equation}
\label{almost_onto}
\left|
F
\left(
\mu_a
\right)
-
(a_0, \ldots, a_{n+1}) 
\right|
\ioe
C
\left|
a_0, \ldots, a_{n+1}
\right|^3
\ 
\text{ with }
\ 
\mu_a
:=
\hat{\mu}_{a_0, \ldots, a_n}
+
\tild{\mu}_{
a_{n+1} - A^{n+1}_K( \hat{\mu}_a)
}
.
\end{equation}
%\begin{equation*}
%\begin{array}{ccccc}
%G_x& : & B(0, \eta) & \to & \R^{n+1} \\
% & & a & \mapsto & a-F( \mu_a)+x. \\
%\end{array}
%\end{equation*}
%%
%
Let $ \rho >0$ such that $C \rho^2 <\frac{1}{2}$ and $x \in \R^{n+2}$ such that $|x | < \frac{\rho}{2}$. Then,  the map 
\begin{equation*}
G_x 
: 
a 
\mapsto 
a
- 
F
\left(
\mu_a
\right)
+x
\end{equation*}
maps the ball $B(0, \rho)$ to itself and is continuous. 
%\begin{align*}
%| G_x(a) | 
%&\ioe 
%\left|
%a
%- 
%F
%\left(
%\hat{\mu}_{a_0, \ldots, a_n}
%+
%\tild{\mu}_{
%a_{n+1} - A^{n+1}_K( \hat{\mu}_a)
%}
%\right)
%\right|
%+
%|x|
%\\
%%&\ioe 
%C |a|^3 + \frac{\rho}{2} 
%\\
%& \ioe 
%\rho
%\end{align*}
%
Therefore, Brouwer's fixed-point theorem entails the existence of $a$ such that $F(\mu_{a})=x$. This holds for every $x$ such that $|x| < \frac{\rho}{2}$. Thus, taking $x=(0, \ldots, 0, \frac{\rho}{4})$, one gets the existence of $\mu^+$ such that $F(\mu^+)= (0, \ldots, 0, \frac{\rho}{4})$. Hence, the function $\mu^+_{n+1}:= \frac{1}{\sqrt{\rho/4}} \mu^+$ satisfies $F( \mu^+_{n+1})=(0, \ldots , 0, 1)$. The function $\mu^-_{n+1}$ is constructed similarly. And this ends the proof. 
\end{proof}

\bibliographystyle{plain}
\bibliography{biblio_quad}

\begin{thebibliography}{10}

\bibitem{BMS82}
John~M. Ball, Jerrold~E. Marsden, and Marshall Slemrod.
\newblock Controllability for distributed bilinear systems.
\newblock {\em SIAM J. Control Optim.}, 20(4):575--597, 1982.

\bibitem{B05}
Karine Beauchard.
\newblock Local controllability of a 1-{D} {S}chr\"{o}dinger equation.
\newblock {\em J. Math. Pures Appl. (9)}, 84(7):851--956, 2005.

\bibitem{B08}
Karine Beauchard.
\newblock Controllability of a quantum particle in a 1{D} variable domain.
\newblock {\em ESAIM Control Optim. Calc. Var.}, 14(1):105--147, 2008.

\bibitem{BLM21}
Karine Beauchard, Jérémy~Le Borgne, and Frédéric Marbach.
\newblock On expansions for nonlinear systems, error estimates and convergence
  issues, 2021.

\bibitem{BL10}
Karine Beauchard and Camille Laurent.
\newblock Local controllability of 1{D} linear and nonlinear {S}chr\"{o}dinger
  equations with bilinear control.
\newblock {\em J. Math. Pures Appl. (9)}, 94(5), 2010.

\bibitem{BM18}
Karine Beauchard and Fr\'{e}d\'{e}ric Marbach.
\newblock Quadratic obstructions to small-time local controllability for
  scalar-input systems.
\newblock {\em J. Differential Equations}, 264(5), 2018.

\bibitem{BM20}
Karine Beauchard and Fr\'{e}d\'{e}ric Marbach.
\newblock Unexpected quadratic behaviors for the small-time local null
  controllability of scalar-input parabolic equations.
\newblock {\em J. Math. Pures Appl. (9)}, 136, 2020.

\bibitem{BM14}
Karine Beauchard and Morgan Morancey.
\newblock Local controllability of 1{D} {S}chr\"{o}dinger equations with
  bilinear control and minimal time.
\newblock {\em Math. Control Relat. Fields}, 4(2), 2014.

\bibitem{BCCS12}
Ugo Boscain, Marco Caponigro, Thomas Chambrion, and Mario Sigalotti.
\newblock A weak spectral condition for the controllability of the bilinear
  {S}chr\"{o}dinger equation with application to the control of a rotating
  planar molecule.
\newblock {\em Comm. Math. Phys.}, 311(2):423--455, 2012.

\bibitem{BCS14}
Ugo Boscain, Marco Caponigro, and Mario Sigalotti.
\newblock Multi-input {S}chr\"{o}dinger equation: controllability, tracking,
  and application to the quantum angular momentum.
\newblock {\em J. Differential Equations}, 256(11):3524--3551, 2014.

\bibitem{B21}
M{\'e}gane Bournissou.
\newblock Local controllability of the bilinear 1{D} {S}chrodinger equation
  with simultaneous estimates.
\newblock {\em arXiv preprint arXiv:2107.08817}, 2021.

\bibitem{BCC20}
Nabile Boussa\"{\i}d, Marco Caponigro, and Thomas Chambrion.
\newblock Regular propagators of bilinear quantum systems.
\newblock {\em J. Funct. Anal.}, 278(6):108412, 66, 2020.

\bibitem{Cer07}
Eduardo Cerpa.
\newblock Exact controllability of a nonlinear {K}orteweg-de {V}ries equation
  on a critical spatial domain.
\newblock {\em SIAM J. Control Optim.}, 46(3):877--899, 2007.

\bibitem{CC09}
Eduardo Cerpa and Emmanuelle Cr\'{e}peau.
\newblock Boundary controllability for the nonlinear {K}orteweg-de {V}ries
  equation on any critical domain.
\newblock {\em Ann. Inst. H. Poincar\'{e} Anal. Non Lin\'{e}aire},
  26(2):457--475, 2009.

\bibitem{CMSB09}
Thomas Chambrion, Paolo Mason, Mario Sigalotti, and Ugo Boscain.
\newblock Controllability of the discrete-spectrum {S}chr\"{o}dinger equation
  driven by an external field.
\newblock {\em Ann. Inst. H. Poincar\'{e} Anal. Non Lin\'{e}aire},
  26(1):329--349, 2009.

\bibitem{CE19}
Shirshendu Chowdhury and Sylvain Ervedoza.
\newblock Open loop stabilization of incompressible {N}avier-{S}tokes equations
  in a 2d channel using power series expansion.
\newblock {\em J. Math. Pures Appl. (9)}, 130:301--346, 2019.

\bibitem{C06}
Jean-Michel Coron.
\newblock On the small-time local controllability of a quantum particle in a
  moving one-dimensional infinite square potential well.
\newblock {\em C. R. Math. Acad. Sci. Paris}, 342(2):103--108, 2006.

\bibitem{C07}
Jean-Michel Coron.
\newblock {\em Control and nonlinearity}, volume 136 of {\em Mathematical
  Surveys and Monographs}.
\newblock American Mathematical Society, Providence, RI, 2007.

\bibitem{CC04}
Jean-Michel Coron and Emmanuelle Cr\'{e}peau.
\newblock Exact boundary controllability of a nonlinear {K}d{V} equation with
  critical lengths.
\newblock {\em J. Eur. Math. Soc. (JEMS)}, 6(3):367--398, 2004.

\bibitem{CKN20}
Jean-Michel Coron, Armand Koenig, and Hoai-Minh Nguyen.
\newblock On the small-time local controllability of a kdv system for critical
  lengths, 2020.

\bibitem{CR17}
Jean-Michel Coron and Ivonne Rivas.
\newblock Quadratic approximation and time-varying feedback laws.
\newblock {\em SIAM J. Control Optim.}, 55(6):3726--3749, 2017.

\bibitem{CRX17}
Jean-Michel Coron, Ivonne Rivas, and Shengquan Xiang.
\newblock Local exponential stabilization for a class of {K}orteweg--de {V}ries
  equations by means of time-varying feedback laws.
\newblock {\em Anal. PDE}, 10(5):1089--1122, 2017.

\bibitem{M18}
Fr\'{e}d\'{e}ric Marbach.
\newblock An obstruction to small-time local null controllability for a viscous
  {B}urgers' equation.
\newblock {\em Ann. Sci. \'{E}c. Norm. Sup\'{e}r. (4)}, 51(5):1129--1177, 2018.

\bibitem{M14}
Morgan Morancey.
\newblock Simultaneous local exact controllability of 1{D} bilinear
  {S}chr\"{o}dinger equations.
\newblock {\em Ann. Inst. H. Poincar\'{e} Anal. Non Lin\'{e}aire},
  31(3):501--529, 2014.

\bibitem{MN14}
Morgan Morancey and Vahagn Nersesyan.
\newblock Global exact controllability of 1{D} {S}chr\"{o}dinger equations with
  a polarizability term.
\newblock {\em C. R. Math. Acad. Sci. Paris}, 352(5):425--429, 2014.

\bibitem{MN15}
Morgan Morancey and Vahagn Nersesyan.
\newblock Simultaneous global exact controllability of an arbitrary number of
  1{D} bilinear {S}chr\"{o}dinger equations.
\newblock {\em J. Math. Pures Appl. (9)}, 103(1):228--254, 2015.

\bibitem{N09}
Vahagn Nersesyan.
\newblock Growth of {S}obolev norms and controllability of the
  {S}chr\"{o}dinger equation.
\newblock {\em Comm. Math. Phys.}, 290(1):371--387, 2009.

\bibitem{NN12bis}
Vahagn Nersesyan and Hayk Nersisyan.
\newblock {Global exact controllability in infinite time of Schr{\"o}dinger
  equation: multidimensional case}.
\newblock {\em {Journal de Math{\'e}matiques Pures et Appliqu{\'e}es}},
  97(4):295--317, April 2012.

\bibitem{N59}
Louis Nirenberg.
\newblock On elliptic partial differential equations.
\newblock {\em Ann. Scuola Norm. Sup. Pisa Cl. Sci. (3)}, 13:115--162, 1959.

\bibitem{P16}
Jean-Pierre Puel.
\newblock Local exact bilinear control of the {S}chr\"{o}dinger equation.
\newblock {\em ESAIM Control Optim. Calc. Var.}, 22(4):1264--1281, 2016.

\bibitem{T00}
Gabriel Turinici.
\newblock On the controllability of bilinear quantum systems.
\newblock In {\em Mathematical models and methods for ab initio quantum
  chemistry}, volume~74 of {\em Lecture Notes in Chem.}, pages 75--92.
  Springer, Berlin, 2000.

\end{thebibliography}

\end{document}